\documentclass[11pt]{amsart}
\usepackage{amsmath,amscd,amssymb,enumerate,verbatim,amsfonts}
\usepackage[T1]{fontenc}
\usepackage{parskip}
\usepackage{palatino}

\newcommand{\mB}{\mathcal{B}}

\newcommand{\mE}{\mathcal{E}}

\newcommand{\mM}{\mathcal{M}}

\newcommand{\HH}{\mathbf{H}}
\newcommand{\CC}{\mathbf{C}}

\newcommand{\RR}{\mathbf{R}}
\newcommand{\QQ}{\mathbf{Q}}
\newcommand{\PP}{\mathbf{P}}
\newcommand{\ZZ}{\mathbf{Z}}

\newcommand{\id}{\mathrm{Id}}

\newcommand{\ev}{\mathrm{ev}}
\newcommand{\Aut}{\mathrm{Aut}}

\newcommand{\Obs}{\mathcal{O}}
\newcommand{\Stab}{\mathrm{Stab}}

\newcommand{\dbar}{\overline{\partial}}
\newcommand{\coker}{\mathrm{coker}}

\newcommand{\dvol}{\mathrm{dvol}}

\newcommand{\GW}{\operatorname{GW}}

\newcommand{\WW}{W}
\newcommand{\torus}{\ZZ^2\backslash\RR^2}
\newcommand{\nilman}{\Gamma\backslash N}
\newcommand{\orient}{\mathfrak{o}}
\newcommand{\nn}{\mathfrak{n}}
\renewcommand{\aa}{\mathfrak{a}}
\newcommand{\zz}{\mathfrak{z}}
\newcommand{\bb}{\mathfrak{b}}
\newcommand{\pp}{\mathfrak{p}}
\newcommand{\qq}{\mathfrak{q}}

\renewcommand{\tt}{\mathfrak{t}}
\renewcommand{\Upsilon}{S}
\newcommand{\sgn}{\mathrm{sgn}}
\newcommand{\orvar}{Q}
\renewcommand{\phi}{\varphi}

\newcommand{\OP}{\operatorname}

\numberwithin{equation}{section}

\newtheorem{thm}[equation]{Theorem}
\newtheorem{lma}[equation]{Lemma}
\newtheorem{prp}[equation]{Proposition}
\newtheorem{cor}[equation]{Corollary}

\theoremstyle{definition}
\newtheorem{ass}[equation]{Assumption}
\newtheorem{dfn}[equation]{Definition}
\newtheorem{exm}[equation]{Example}

\theoremstyle{remark}

\newtheorem{rmk}[equation]{Remark}

\title[Pseudoholomorphic tori in Kodaira-Thurston]{Pseudoholomorphic tori \\ in the Kodaira-Thurston manifold}
\author{Jonathan David Evans}
\email{j.d.evans@ucl.ac.uk}
\address{Department of Mathematics, University College London, Gower Street, London, WC1E 6BT}
\author{Jarek K\k{e}dra}
\email{kedra@abdn.ac.uk}
\address{University of Aberdeen and University of Szczecin}

\begin{document}
\begin{abstract}
The Kodaira-Thurston manifold is a quotient of a nilpotent Lie group by a cocompact lattice. We compute the family Gromov-Witten invariants which count pseudoholomorphic tori in the Kodaira-Thurston ma\-ni\-fold. For a fixed symplectic form the Gromov-Witten invariant is trivial so we consider the twistor family of left-invariant symplectic forms which are orthogonal for some fixed metric on the Lie algebra. This family defines a loop in the space of symplectic forms. This is the first example of a genus one family Gromov-Witten computation for a non-K\"ahler manifold.
\end{abstract}
\maketitle


\section{Introduction}

The enumerative geometry of complex curves in complex manifolds is an old and venerable pursuit \cite{Klei} which gained momentum in the last fifteen years of the twentieth century for two main reasons. The first of these was Gromov's paper \cite{Gro} which explained how to count persistent pseudoholomorphic curves in symplectic manifolds and gave many applications of the existence of pseudoholomorphic curves to symplectic topology. The second was Witten's recasting \cite{W} of Gromov's theory in the language of topological sigma models (Gromov-Witten theory, see \cite{RT1,RT2} for a mathematical approach) and the subsequent observation \cite{COGP} of Candelas-de la Ossa-Green-Parks that string dualities give concrete predictions in enumerative geometry for rational (genus zero) curves in Calabi-Yau three-folds. In the case of higher genus curves, predictions were made by Bershadksy, Cecotti, Ooguri and Vafa \cite{BCOV}. In many places these predictions have been confirmed by computations: see for example \cite{Z} in the genus one case.

The enumerative invariants which emerge have beautiful structural properties, for instance Bryan and Leung \cite{BL} computed the Gromov-Witten invariants for the hyper-K\"{a}hler sphere of K3 surfaces (a family Gromov-Witten invariant, see Section \ref{famGW} below) and showed that the generating function for these numbers (the {\em genus $g$ Gromov-Witten potential}) is

\[\left(q\frac{d}{dq}G_2(q)\right)^g/\Delta(q)\]
where $G_2(q)=-\frac{1}{24}+\sum_{n\geq 1}\sigma_1(n)q^n$, $\sigma_1(n)=\sum_{d|n}d$ is the sum of divisors of $n$ and $\Delta(q)=q\prod_{n\geq 1}(1-q^n)^{24}$. These are the weight 2 Eisenstein series quasimodular form and the discriminant modular form respectively.

The physical predictions make use of the special geometry of Calabi-Yau manifolds, but Gromov's philosophy\footnote{``What fascinated me even more was the familiar web of algebraic curves in a surface emerging in its full beauty in the softish environment of general (nonintegrable!) almost complex structures. (Integrability had always made me feel claustrophobic.)'' \cite{GromAMS}} is that integrability of an almost complex structure is not necessary to have an intricate enumerative theory of holomorphic curves. One might ask what physical predictions hold in the world of non-K\"{a}hler symplectic manifolds and whether similarly beautiful formulae can be found for their Gromov-Witten potentials. One large and well-understood class of symplectic manifolds containing non-K\"{a}hler examples are the symplectic nilmanifolds: compact left-quotients of a nilpotent Lie group equipped with a left-invariant symplectic form. Apart from the higher-dimensional tori these are all non-K\"{a}hler \cite{BG}; the best known is the four-dimensional Kodaira-Thurston manifold \cite{Th} which will be the focus of this paper.

At the time of writing, mirror symmetry is not known for the Kodaira-Thurston manifold. In particular, it is not clear to the authors what the genus one partition function of the mirror B-model might be. It would be interesting to compute this and compare with the Gromov-Witten invariants we calculate here, as a first step towards understanding the conjectures of \cite{BCOV} in a non-K\"{a}hler setting.

The Gromov-Witten invariants of a single symplectic nilmanifold are not very interesting. In many cases (like the Kodaira-Thurston manifold) one can connect a left-invariant symplectic form $\omega$ to $-\omega$ along a path of left-invariant symplectic forms. Since Gromov-Witten invariants are invariant under deformations of the symplectic form, if there is a non-zero homology class $A$ with a non-zero Gromov-Witten invariant then there is a $J_+$-holomorphic curve representing $A$ (where $J_+$ is $\omega$-compatible) and a $J_-$-holomorphic curve representing $A$ (where $J_-$ is $-\omega$-compatible). Since non-constant $J$-holomorphic curves have positive $\omega$-area, it follows that $\int_{A}\omega$ is both positive and negative, a contradiction.

However, if we allow families of symplectic manifolds and of almost complex structures, and if we count holomorphic curves which are holomorphic for some $J$ in the family, then we can still obtain non-zero invariants. These {\em family Gromov-Witten invariants} have been defined and computed in various places in the literature, for instance the Bryan-Leung computation mentioned above; we recall the definition in Section \ref{famGW}. Moreover there are certain natural families of left-invariant symplectic forms on nilmanifolds (the {\em twistor families}). This paper provides techniques for computing the genus one family Gromov-Witten invariants for the twistor families of symplectic nilmanifolds. We do the complete computation for the Kodaira-Thurston manifold.

\begin{dfn}[Kodaira-Thurston manifold]
Let $\Gamma$ be the group of affine transformations of $\RR^4$ (with coordinates $x,y,z,t$) generated by three translations (in the $y$, $z$ and $t$ directions) and the map
\[x\mapsto x+1,\ y\mapsto y,\ z\mapsto z+y,\ t\mapsto t\]
The quotient $\Gamma\backslash\RR^4$ is the {\em Kodaira-Thurston manifold}, $K$. It is also a left-quotient of a simply-connected nilpotent Lie group $N$ by a cocompact lattice $\Gamma\subset N$; see Section \ref{KT}.
\end{dfn}
The twistor family for $K$ is the circle of symplectic forms
\[\omega_{\theta}=dt\wedge a_{\theta}+(dz-xdy)\wedge b_{\theta}\]
where $a_{\theta}=\cos\theta\ dx+\sin\theta\ dy$ and $b_{\theta}=-\sin\theta\ dx+\cos\theta\ dy$. The second integral homology $H_2(K;\ZZ)$ is of rank 4 and is generated by the homology classes represented by the tori
\[E_{ij}=[\{(x_1,x_2,x_3,x_4)\in K\ |\ x_i=x_j=0\}]\]
where $i\in\{1,2\}$ and $j\in\{3,4\}$. We also define $E_{i34}\in H_3(K;\ZZ),\ i\in\{1,2\}$ by
\[E_{i34}=[\{(x_1,x_2,x_3,x_4)\in K\ |\ x_i=0\}].\]
See Section \ref{coh} for a complete description of the homology and cohomology. We will write
\[[A_{13},A_{23},A_{14},A_{24}]\]
for the homology class $A=\sum A_{ij}E_{ij}$. By Lemma \ref{pluck} we know that if $A\in H_2(K;\ZZ)$ is a homology class represented by a torus then $A_{13}A_{24}=A_{14}A_{23}$. It is helpful to remember that $K$ is a $T^2$-bundle over $T^2$ in two ways:
\begin{itemize}
\item There is a projection $(x,y,z,t)\mapsto (x,y)$ whose fibres are Lagrangian with respect to $\omega_{\theta}$ for all $\theta$.
\item There is a projection $(x,y,z,t)\mapsto (x,t)$ whose fibres are symplectic for $\omega_0$.
\end{itemize}

\begin{thm}\label{ourmainthm}
Let $K$ denote the Kodaira-Thurston manifold and $\WW$ the twistor family of left-invariant symplectic structures on $K$. If $A=[A_{13},A_{23},A_{14},A_{24}]\in H_2(K;\ZZ)$ is a non-zero homology class and $m=\gcd(A_{13},A_{23})$, $n=\gcd(A_{14},A_{24})$ then
\[\GW_{1,1}(\WW,A)=\frac{(m^2+n^2)\sigma_2(\gcd(m,n))}{\gcd(m,n)^3}(A_{13}E_{134}+A_{23}E_{234})\in H_3(K;\ZZ)\]
where $\sigma_2(x)=\sum_{d|x}d^2$.
\end{thm}

Note that, according to our definition (Equation \eqref{dfn:gw}), family Gromov-Witten invariants should really live in $H_*(K\times W;\ZZ)$. In our situation they all have the form $C\otimes[\star]$ where $[\star]\in H_0(W,\ZZ)$ is the homology class of a point. This is because each connected component of the moduli space of tori which are $J_{\theta}$-holomorphic for some $\theta\in\WW$ consists of tori which are $J_{\theta}$-holomorphic for a fixed $\theta$.

We would like to stress that the moduli spaces of pseudoholomorphic tori we consider are, unusually, odd-dimensional. This can be understood as follows. The index of the Fredholm problem for counting (unmarked) tori in a four-manifold with $c_1=0$ is zero. In each space of $\omega_{\theta}$-compatible almost complex structures there is a codimension one `wall' of almost complex structures where the kernel of the Fredholm problem is one-dimensional and there is a one-dimensional cokernel. For all the moduli spaces which contribute to the Gromov-Witten invariant, our one-dimensional family $J_{\theta}$ is transverse to that wall, and so the (one-dimensional) moduli space is regular from the point of view of family Gromov-Witten theory.

Let us spell out the geometric content of this theorem for the specific family~$J_{\theta}$. Since the Gromov-Witten class is in $H_3(K;\ZZ)$ it detects holomorphic curves intersecting a loop in $K$. Let $L$ be a loop in $K$ and let $A\in H_2(K;\ZZ)$ be as in the statement of the theorem. There is a unique almost complex structure $J_{\theta} \in \WW$ (see Lemma 5.7) for which there are $J_{\theta}$-holomorphic tori representing $A$. For this $J_{\theta}$ the pseudoholomorphic tori intersect $L$ at
\[\frac{m^2+n^2}{\gcd(m,n)^3}\sigma_2(\gcd(m,n))(A_{13}E_{134}+A_{23}E_{234}) \cap [L]\]
points (counted with multiplicity and signs). The complex structure on the domain torus is allowed to vary but, when $m\neq 0$ it is actually constant over each component of the moduli space (Lemma 5.7 again).

Two obvious classes containing holomorphic tori are $E_{13}$ and $E_{14}$.
\begin{exm}
The class $E_{14}$ is represented by the fibres of the projection
\[(x,y,z,t)\mapsto (x,t).\]
These are $J_0$-holomorphic tori and all $J_0$-holomorphic tori have this form. In this case $m=0$ and the $j$-invariant of these fibres varies in a loop. In fact, these tori are all irregular and do not contribute to the Gromov-Witten invariant, see Lemma~\ref{obscalc}.
\end{exm}
\begin{exm}
The class $E_{13}$ is represented by sections of the projection
\[(x,y,z,t)\mapsto (x,t).\]
There is an $S^1$-family of $J_{\pi/2}$-holomorphic sections which intersect a loop $L$ in $[L]\cap [E_{134}]$ points (with multiplicity). In fact, we will not deal with these sections explicitly because they do not correspond to \emph{reduced Lie algebra homomorphisms} (see Definition \ref{reddfn}). Instead we will apply a diffeomorphism of $K$ to move to the homology class $E_{13}+E_{23}$ without affecting the Gromov-Witten computation, using Lemma \ref{wlog} and Equation \eqref{gwinvariance}. The $J_{\pi/2}$-holomorphic tori representing $E_{13}$ become $J_{\pi/4}$-holomorphic tori representing $E_{13}+E_{23}$.
\end{exm}

\subsection*{Outline of proof of Theorem \ref{ourmainthm}}

Theorem \ref{ourmainthm} is proved by reducing the problem to the enumeration of certain homomorphisms $\ZZ^2\to\pi_1(K)$. Let $u\colon T^2\to K$ be a $J$-holomorphic torus where $J$ is a left-invariant almost complex compatible with $\omega$. We perform the following steps:
\begin{itemize}
\item Take a lift of $u$ to the universal covers $\tilde{u}\colon\RR^2\to N$ and compare it with the unique Lie group homomorphism $H\colon\RR^2\to N$ extending the induced map $\pi_1(u)\colon\ZZ^2\to\pi_1(K)$ on fundamental groups.
\item The map $(p,q)\mapsto H(p,q)^{-1}\tilde{u}(p,q)$ is then bounded (Corollary \ref{bdd}) and we are interested in its logarithm $C\colon\RR^2\to\nn$ where $\nn$ is the Lie algebra of $N$.
\item The Cauchy-Riemann equations for $C$ imply that $C$ satisfies a second order elliptic system of equations. In Proposition \ref{classiftori} we show that this system separates into equations for which the Hopf maximum principle holds \cite[Theorem 3.1]{GT}. We apply this to prove that $C$ is constant. Hence, for all left-invariant $J$ which are compatible with a left-invariant symplectic form, all $J$-holomorphic tori are of the form $ve^C$, where $v$ comes from a Lie algebra homomorphism and $C$ is a constant.
\item Another maximum principle allows us to study the linearised problem. We prove that all moduli spaces are cut out cleanly by the Cauchy-Riemann operator (Theorem \ref{cleanthm}). In Section \ref{regobs} we determine which tori are regular and, for non-regular tori, we write down an explicit section of the obstruction bundle, proving that these do not contribute to the Gromov-Witten invariant. In Section \ref{orient}, we determine orientations on the moduli spaces.
\item It remains to count the tori. By applying an automorphism of $\Gamma$ we reduce ourselves to considering only non-zero homology classes
\[[A_{13},A_{23},A_{14},A_{24}]\]
where $A_{13}=A_{23}$ and $A_{14}=A_{24}$ (Lemma \ref{wlog}), for which we can further assume that the homomorphism $H$ has a particularly simple form (Lemma \ref{reparawlog}). This enables us to enumerate the tori and to understand the homology classes represented by the evaluation maps (Section \ref{enumtori}).
\end{itemize}

\subsection*{Generalisations}

We carry out the full calculation only for the Kodaira-Thurston manifold $K$ but we formulate the problem for two-step symplectic nilmanifolds in general. The universal cover of $K$ is the only nonabelian symplectic two-step nilpotent Lie group in dimension four. For examples in higher dimensions the main difference is that the Cauchy-Riemann equations are a more complicated elliptic system (with more serious nonlinearities and coupling) and it becomes harder to apply the maximum principle. Our methods can be extended to a limited range of homology classes in certain higher-dimensional examples - the ones constructed in \cite{CFL}. For $k$-step nilmanifolds with $k\geq 3$ the equations are even harder to deal with.

\subsection*{Outline of the paper}
\begin{itemize}
\item In Section \ref{back} we explain the classical computations of genus one Gromov-Witten invariants for two-tori and for higher-dimensional tori. We also define the family Gromov-Witten invariants.
\item In Section \ref{nilp} we introduce two-step nilpotent Lie groups and their twistor families $\WW$ of symplectic structures. We also write down the Cauchy-Riemann equations for the logarithm of a $J$-holomorphic torus ($J$ compatible with some $\omega\in\WW$).
\item In Section \ref{KT} we review the Kodaira-Thurston manifold and its basic properties. In particular we show that all tori are descended from right-translates of Lie algebra homomorphisms (Proposition \ref{classiftori}).
\item In Section \ref{holtor} we describe the moduli spaces of holomorphic tori.
\item In Section \ref{auto} we compute the automorphism groups of the unmarked holomorphic tori in the Kodaira-Thurston manifold.
\item In Section \ref{lintheory} we study the linearised problem, including checking regularity and computing obstruction bundles and orientations.
\item In Section \ref{enumtori} we complete the proof of Theorem \ref{ourmainthm}.
\end{itemize}


\subsection*{Notation}

For brevity in our coordinate expressions we will sometimes use the following convention to denote antisymmetrisation of indices:
\[A_{[ij]}=A_{ij}-A_{ji}\]
for example:
\[\partial_{[p}X_i\partial_{q]}Y_j=\partial_pX_i\partial_qX_j-\partial_qX_i\partial_pX_j.\]


\section{Background}\label{back}

We begin by giving an overview of genus one Gromov-Witten theory for the (twistor families of) complex tori, which are precisely the nilmanifolds arising from an abelian Lie group. In doing so we build up in embryonic form many of the ideas we need for the nonabelian case.


\subsection{The $2$-torus}\label{2torus}

The space whose genus one Gromov-Witten invariants are easiest to calculate is the two-torus, $T^2$. Let $\Lambda_{\tau}\cong\ZZ^2$ denote the $\ZZ$-lattice in $\CC$ spanned by the vectors $1$ and $\tau=\tau_1+i\tau_2$, $\tau_2>0$. Let $\Sigma_{\tau}=\Lambda_{\tau}\backslash\CC$ denote the corresponding complex torus.
\begin{lma}
Any non-constant holomorphic map $f\colon \Sigma_{\tau'}\to \Sigma_{\tau}$ is a covering map (unbranched).
\end{lma}
This is clear because branching increases genus by the Riemann-Hurwitz formula. Therefore counting holomorphic maps $\Sigma_{\tau'}\to \Sigma_{\tau}$ of degree $\ell\geq 1$ amounts to counting holomorphic covering spaces of $\Sigma_{\tau}$ or, equivalently, sublattices $\Lambda_{\tau'}\subset\Lambda_{\tau}$ of index $\ell$ modulo the action of $\OP{SL}(2,\ZZ)$ which re\-pa\-ra\-metri\-ses the domain.
\begin{lma}
There are $\sigma_1(\ell)=\sum_{d|\ell}d$ sublattices of $\Lambda_{\tau}$ of index $\ell\geq 1$, modulo the action of $\OP{SL}(2,\ZZ)$.
\end{lma}
\begin{proof}
This is standard and we reproduce the argument only for comparison later. A sublattice of index $\ell$ is specified by a homomorphism $\ZZ^2\to\ZZ^2$ whose image has index $\ell$, that is a two-by-two integer matrix
\[\left(\begin{array}{cc}
a & b\\
c & d
\end{array}\right)\]
with determinant $\ell$. The $\OP{SL}(2,\ZZ)$-action is just right multiplication. Using this right action one can perform the Euclidean algorithm on $c$ and $d$ to ensure that $c$ vanishes. Similarly one can ensure that $0\leq b< a$. Now for each $d|\ell$ there are $d$ possible matrices up to the action of $\OP{SL}(2,\ZZ)$
\[\left(\begin{array}{cc}
d & b\\
0 & \ell/d
\end{array}\right),\ b=0,\ldots,d-1.\]
\end{proof}
We define the moduli space
\[\mM_{1,1}(\Sigma_{\tau},\ell[\Sigma_{\tau}])\]
to consist of equivalence classes of pairs $(u,z)$ where $u\colon\Sigma_{\tau'}\to\Sigma_{\tau}$ is a holomorphic map of degree $\ell\geq 1$ and $z\in\Sigma_{\tau'}$ is a point. The equivalence relation equates $(u,z)$ with $(u\circ\phi^{-1},\phi(z))$ for any holomorphic automorphism $\phi\colon\Sigma_{\tau'}\to\Sigma_{\tau'}$ for which $u\circ\phi^{-1}=u$. This has an evaluation map
\[\ev\colon\mM_{1,1}(\Sigma_{\tau},\ell[\Sigma_{\tau}])\to \Sigma_{\tau},\ [u,z]\mapsto u(z)\]
\begin{lma}
The map $\ev$ has degree $\sigma_1(\ell)$.
\end{lma}
\begin{proof}
We have seen that there are $\sigma_1(\ell)$ tori in the moduli space and that all of these are $\ell$-fold covering spaces of $\Sigma_{\tau}$. Fix one such covering map. If $x\in\Sigma_{\tau}$ then the preimages of this point under this covering map are all equivalent by the action of the deck group, which acts by holomorphic automorphisms preserving the covering. Hence they represent the same element in the moduli space $\mM_{1,1}(\Sigma_{\tau},\ell[\Sigma_{\tau}])$, so the degree of the evaluation map is just the number of covering spaces, $\sigma_1(\ell)$.
\end{proof}
These curves are all regular in the sense of Gromov-Witten theory: the cokernel of the linearisation is just the quotient of the Dolbeault cohomology group $H^{0,1}(\Sigma_{\tau'};u^*T\Sigma_{\tau})$ by the image of $H^{0,1}(\Sigma_{\tau'};T\Sigma_{\tau'})$ under pushforward $du\colon T\Sigma_{\tau'}\to T\Sigma_{\tau}$ (this quotient corresponds to allowing $\tau'$ to vary)
\[H^{0,1}(\Sigma_{\tau'};u^*T\Sigma_{\tau})/H^{0,1}(\Sigma_{\tau'};T\Sigma_{\tau'})\cong H^{0,1}(\Sigma_{\tau'};T\Sigma_{\tau'})/H^{0,1}(\Sigma_{\tau'};T\Sigma_{\tau'})=0\]
The 1-point Gromov-Witten class of degree $\ell\geq 1$, genus one curves through a point of $T^2$ is the pushforward under $\ev$ of the fundamental class of the moduli space and is therefore given by
\[\GW_{1,1}(T^2,\ell[T^2])=\sigma_1(\ell)[T^2].\]


\subsection{The $2n$-torus}

The situation for the $2n$-torus is similar but some of the features it presents are new and will be developed in a more general context later in the paper. For a start, a generic abelian variety contains no closed holomorphic curves, so we know that the Gromov-Witten invariants vanish. However, holomorphic curves persist if we take a family of complex structures and look for curves which are holomorphic with respect to one of the complex structures. This phenomenon, made precise in Section \ref{famGW}, is familiar from the case of K3 surfaces \cite{BL}, where an elliptically-fibred K3 contains elliptic curves through every point which disappear if one perturbs the complex structure, but which persist in families which are deformations of the hyper-K\"{a}hler two-sphere of complex structures. The Gromov-Witten invariants which count curves representing some second homology class $A\neq 0$ which are $J$-holomorphic for some $J$ in a fixed finite-dimensional, compact, oriented family are called \emph{family Gromov-Witten invariants}. Note that these $J$ must all be tamed by symplectic forms in order to achieve Gromov compactness, but that the taming form (and even its cohomology class) might depend on $J$.

A natural generalisation of this hyper-K\"{a}hler sphere to examples which are not hyper-K\"{a}hler is the following.

\begin{dfn}[Twistor family]\label{torustwist}
Let $g$ be an inner product on the vector space $\RR^{2n}$ and let $\orient$ be an orientation. The \emph{twistor family} of complex structures is the space of $\orient$-positive orthogonal complex structures
\[\WW=\{\psi\in GL^+(\RR^{2n})\ |\ \psi^2=-\id,\ g(\psi X,\psi Y)=g(X,Y)\mbox{ for all }X,Y\in\RR^{2n}\}\]
Note that $\WW\cong SO(2n)/U(n)$ since $SO(2n)$ acts transitively on $\WW$ with stabiliser $U(n)$. Each $\psi\in\WW$ gives rise to a two-form
\[\omega_{\psi}(X,Y):=-g(X,\psi Y)\]
and to a bi-invariant K\"{a}hler structure $(\Omega_{\psi},J_{\psi})$ on the torus $\ZZ^{2n}\backslash\RR^{2n}$.
\end{dfn}

The genus one family Gromov-Witten invariants of the twistor family of $2n$-tori are easy to compute. Let us write $\GW_{1,k}(\WW,A)\in H_*((T^{2n})^k\times\WW;\ZZ)$ for the homology class of the evaluation pseudocycle for genus one curves representing the homology class $A\neq 0$ which are $J_{\psi}$-holomorphic for some $\psi\in~\!\WW$ (see Section \ref{famGW} for definitions).

Note first that if $\phi\in\OP{SL}(2n,\ZZ)$ is a matrix then $\phi^*\WW$ is the twistor family of $\phi^*g$. Since $g$ and $\phi^*g$ can be connected by a path of inner products we see that $\phi^*\WW$ and $\WW$ are isotopic as families of complex structures. More importantly, the families $\{\omega_{\psi}\}_{\psi\in\WW}$ and $\{\omega_{\phi^*\psi}\}_{\psi\in\WW}$ of taming symplectic forms are isotopic. The family Gromov-Witten invariants are equivariant under diffeomorphisms $\phi$, so that
\[\phi_*\GW_{1,k}(\WW,A)=\GW((\phi^{-1})^*(\WW),\phi_*A),\]
and also unchanged by deformations through tamed families, hence we see that
\[\GW_{1,k}(\WW,\phi_*A)=\phi_*\GW_{1,k}(\WW,A).\]
Homology classes represented by two-tori are specified by homomorphisms $\rho\colon\ZZ^2\to\ZZ^{2n}$ on the level of fundamental groups; we write $[\rho]$ for the corresponding homology class. Two such homomorphisms $\rho$ and $\rho'$ define the same homology class if and only if $\Lambda^2\rho=\Lambda^2\rho'$, that is if all two-by-two minors of $\rho$ and $\rho'$ agree. Acting on the left by an element of $\OP{SL}(2n,\ZZ)$ one can assume that
\[\rho=\left(\begin{array}{cc}
\rho_{11} & \rho_{12}\\
\rho_{21} & \rho_{22}\\
0 & 0\\
\vdots & \vdots\\
0 & 0
\end{array}\right)\]
The counting of such homomorphisms up to the reparametrisation action of $\OP{SL}(2,\ZZ)$ is again performed by the function $\sigma_1(\ell)$ where $\ell$ is the only nonvanishing two-by-two minor, so $\ell=\ell([\rho])$ is the divisibility of the homology class (the only invariant of the $\OP{SL}(2n,\ZZ)$-action).

Each homomorphism $\rho\colon\ZZ^2\to\ZZ^{2n}$ actually defines a 2-plane $\Pi(\rho)\subset\RR^{2n}$ which is $J_{\psi}$-holomorphic for $\psi$ in a subvariety $\WW(\rho)\subset\WW$. This subvariety is diffeomorphic to $SO(2n-2)/U(n-1)$. Each such 2-plane descends to a $J_{\psi}$-holomorphic genus one curve $v\colon T^2\to T^{2n}$ in $T^{2n}$.
\begin{lma}
All $J_{\psi}$-holomorphic curves in the homology class $[\rho]$ are affine translates of $v$.
\end{lma}
\begin{proof}
This is Liouville's theorem. Let $(a,b)$ be conformal coordinates on $T^2$ and let $\tilde{u}\colon\RR^2\to\RR^{2n}$ denote the lift of an arbitrary $J_{\psi}$-holomorphic curve $u$ in the homology class $[\rho]$ to the universal cover. The Cauchy-Riemann equations are linear and the complex structure is constant
\[\partial_b\tilde{u}=\psi\partial_a\tilde{u}\]
hence in each coordinate of $\RR^{2n}$ the Laplacian $\Delta \tilde{u}_i=\partial_a^2\tilde{u}_i+\partial_b^2\tilde{u}_i=0$. We also have $\Delta h=0$ where $h$ is the inclusion of $\Pi\subset\RR^{2n}$.

It is easy to see that $u$ and $v$ are homotopic and hence $\tilde{u}-h$ is bounded and harmonic. By the maximum principle it is constant which proves that the two curves are affine translates of one another.
\end{proof}

This implies there is precisely one $J_{\psi}$-holomorphic curve in the class $[\rho]\neq 0$ through each point for any $\psi\in\WW(\rho)$. Therefore the family Gromov-Witten invariant is
\[\GW_{1,1}(\WW,[\rho])=\sigma_1(\ell([\rho]))[T^{2n}]\otimes[\WW(\rho)]\in H_*(T^{2n}\times\WW;\ZZ).\]


\subsection{Family Gromov-Witten invariants}\label{famGW}


Family Gromov-Witten invariants have been defined, calculated and used in many places in the literature including \cite{Buse,KedOno,KlemmEtAl,LeOno,Lee,Lee2,LeeLeung,LeeParker,Lu,MauPan,Sei}. Below, we explain the special cases we require. For more details see \cite{MS04} and \cite{RT2}.

\subsubsection{Setting}

We first set up some notation and assumptions for the rest of this section.

\begin{ass}\label{ass:1}
Let $X$ be a compact, connected, smooth, oriented manifold. Let $\Omega$ denote the space of symplectic forms on $X$. Let $B$ be a compact, smooth, oriented manifold and $\omega$ be a \emph{family of symplectic structures on $X$}, that is a map $\omega\colon B\to\Omega$. We will assume that $(X,\omega(b))$ is a symplectically aspherical symplectic manifold with $c_1=0$. We will denote by $A\in H_2(X;\ZZ)$ a non-zero homology class.
\end{ass}

\begin{rmk}
Note that the Kodaira-Thurston manifold, which is our main example, is a quotient of a nilpotent Lie group $N$ by a cocompact discrete subgroup equipped with a left-invariant symplectic form. All such examples are aspherical and satisfy $c_1=0$. We will specify the family $\omega$ in Definition \ref{dfn:twist}.
\end{rmk}

Let $\mathcal{J}$ denote the space of almost complex structures on $X$.

\begin{dfn}
A family of $\omega$-compatible almost complex structures $J$ is a map
\[J\colon B\to\mathcal{J}\]
such that $J(b)$ is $\omega(b)$-compatible for all $b\in B$. We will write $\mathcal{J}(B)$ for the space of families of $\omega$-compatible almost complex structures.
\end{dfn}

\subsubsection{Complex structures on the torus}\label{sct:torus}

Let $(p,q)$ be coordinates on $\RR^2$ and
\[j_{\tau}=\left(\begin{array}{cc}
-\frac{\tau_1}{\tau_2} & -\left(\frac{\tau_1^2}{\tau_2}+\tau_2\right)\\
\frac{1}{\tau_2} & \frac{\tau_1}{\tau_2}
\end{array}\right)\]
be a complex structure, where $\tau=\tau_1+i\tau_2$ is an element of
the upper half-plane $\HH\subset \CC$. This descends to the quotient
$\torus$ and gives a complex torus $\Sigma_{\tau}$. Equivalently
we can consider coordinates $(a,b)$ on $\RR^2$ with the complex
structure
\[j_i=\left(\begin{array}{cc}
0 & -1\\
1 & 0
\end{array}\right)\]
and divide by the lattice $\Lambda_{\tau}=\langle1,\tau\rangle$, i.e.
\[\Sigma_{\tau}=(\torus,j_{\tau})\stackrel{\Phi_{\tau}}{\cong}(\Lambda_{\tau}\backslash\RR^2,j_{i})\]
where the diffeomorphism is
\[\Phi_{\tau}\left(\begin{array}{c}p \\ q\end{array}\right)=\left(\begin{array}{cc}
1 & \tau_1\\
0 & \tau_2
\end{array}
\right)\left(\begin{array}{c}
p\\
q
\end{array}\right)=\left(\begin{array}{c}
a\\
b
\end{array}\right)\]
or
\[\Phi_{\tau}^{-1}\left(\begin{array}{c}a \\ b\end{array}\right)=\left(\begin{array}{cc}
1 & -\frac{\tau_1}{\tau_2}\\
0 & \frac{1}{\tau_2}
\end{array}
\right)\left(\begin{array}{c}
a\\
b
\end{array}\right)=\left(\begin{array}{c}
p\\
q
\end{array}\right).\]

\subsubsection{Moduli space of pseudoholomorphic maps}

\begin{dfn}
Given a non-zero homology class $A\in H_2(X;\ZZ)$ and a family of compatible almost complex structures $J\in\mathcal{J}(B)$, define the space
\[\mM_{1,1}(A,J)\]
consisting of equivalence classes of quadruples $(u,\tau,z,b)$ where $\tau\in\HH$, $b\in B$, $z\in\torus$ is a marked point and $u$ is a $(j_{\tau},J(b))$-holomorphic map
\[u\colon\torus\to X,\qquad d_zu(j_{\tau}v)=J(b)d_zu(v)\]
such that $u_*([\torus])=A$. We say that two quadruples are equivalent $(u,\tau,z,b)\sim(u',\tau',z',b')$ if there exists a diffeomorphism $\varphi\colon\torus\to\torus$ such that
\[b=b',\ u'=u\circ\varphi^{-1},\ z'=\varphi(z),\ \ j_{\tau'}=\varphi^*j_{\tau}.\]
Note that, since $g=1$, $c_1=0$ and there is one marked point, the expected dimension of this moduli space is $\dim(B)+2$. There is also a well-defined evaluation map
\[\ev\colon\mM_{1,1}(A,J)\to X\times B,\qquad\ev(u,z,\tau,b)=(u(z),b).\]
\end{dfn}
\begin{rmk}[Compactness]\label{rmk:compactification}
To compactify the moduli space of genus one curves with one marked point we consider the moduli space $\overline{M}_{1,1}(A,J)$ of genus one stable maps to $X$ with one marked point. If $X$ is symplectically aspherical then the domain of a stable map in $\overline{M}_{1,1}(A,J)$ is necessarily an irreducible smooth genus one curve, so the moduli space $\mM_{1,1}(A,J)$ is already compact. To see this, note that if the domain is nodal then there is a sphere component with precisely two special points (that is points which are either marked or nodal). By stability, this sphere component must be non-constant, but since we are assuming $X$ to be symplectically aspherical a stable map can have no non-constant sphere components.
\end{rmk}
If $J$ is regular (see Definition \ref{dfn:regul} below) then $\mM_{1,1}(A,J)$ is a smooth, compact, oriented, $(\dim B+2)$-dimensional manifold and we can define the Gromov-Witten invariant to be the homology class
\begin{equation}\label{dfn:gw}\OP{GW}_{1,1}(\omega,A)=\ev_*([\mM_{1,1}(A,J)])\in H_*(X\times B;\ZZ)\end{equation}
which is equivariant under diffeomorphisms $\phi$ of $X$:
\begin{equation}\label{gwequi}\phi_*\GW_{1,k}(\omega,A)=\GW_{1,k}((\phi^{-1})^*\omega,\phi_*A).\end{equation}
To define genus one Gromov-Witten invariants properly \cite{RT2} one must study moduli spaces of solutions to the perturbed Cauchy-Riemann equations for a suitable perturbation $\nu$ depending on $z\in\Sigma$ and $j_{\tau}$. We omit further discussion of the general definition because in all our examples, pseudoholomorphic curves are either regular or can be made regular by a perturbation of $J\in\mathcal{J}(B)$.

\subsubsection{Regularity and obstructions}\label{sct:regul}
Let $\mB$ denote the $W^{1,p}$-completion of the space of smooth maps $u\colon\torus\to X$. There is a Banach bundle $\mE$ over $B\times\HH\times\mB$ whose fibre at $(b,\tau,u)$ is the $L^p$-completion
\[L^p\Omega^{0,1}_{j_{\tau},J(b)}(\Sigma,u^*TX).\]
There is a section $\dbar\colon B\times\HH\times\mB\to\mE$ given by
\[\dbar(b,\tau,u)=J(b)du-du\circ j_{\tau}.\]
If $(b,\tau,u)\in\dbar^{-1}(0)$ then $u$ is a $(j_{\tau},J(b))$-holomorphic curve and the section has a natural vertical linearisation
\[D_{(b,\tau,u)}\dbar\colon T_bB\times T_{\tau}\HH\times W^{1,p}(\Sigma,u^*TX)\to L^p\Omega^{0,1}_{j_{\tau},J(b)}(\Sigma,u^*TX)\]
called the linearised Cauchy-Riemann operator.
\begin{dfn}[Regularity]\label{dfn:regul}
We say that a family $J$ of $\omega$-compatible almost complex structures is \emph{regular} if for every $(b,\tau,u)\in\dbar^{-1}(0)$ (with $u$ simple or multiply-covered) the linearised Cauchy-Riemann operator $D_{(b,\tau,u)}\dbar$ is surjective. Equivalently, the section $\dbar$ vanishes transversely.
\end{dfn}
If $J$ is regular then we can define Gromov-Witten invariants by Equation \eqref{dfn:gw}. More generally we can compute Gromov-Witten invariants using a $J$ which is not regular but for which $\dbar$ vanishes cleanly.
\begin{dfn}[Cleanliness]\label{clean}
We say that a family $J$ of $\omega$-compatible almost complex structures is \emph{clean} if, at every point $(b,\tau,u)\in\dbar^{-1}(0)$ (with $u$ simple or multiply-covered) the moduli space $\dbar^{-1}(0)$ is a smooth manifold with tangent space $\ker(D_{(b,\tau,u)}\dbar)$. Equivalently, $\dbar$ vanishes cleanly. In this case the cokernels $\coker(D\dbar)$ form a vector bundle over $\dbar^{-1}(0)$ which we call the \emph{obstruction bundle} and denote $\Obs$.
\end{dfn}
The following theorem can be proved by a simple modification of the proof of {\cite[Proposition 7.2.3]{MS04}}. The key point is that, since $X$ is symplectically aspherical, there are no nodal genus one stable maps with one marked point, so $\mM_{1,1}(A,J)$ is compact (see Remark \ref{rmk:compactification}).
\begin{thm}
Let $(X,\omega)$ be as in Assumption \ref{ass:1}. If $J$ is a clean family of $\omega$-compatible almost complex structures then the one-point Gromov-Witten invariant is given by
\[\GW_{1,1}(\omega,A)=\ev_*\OP{PD}(e(\Obs))\]
where $\OP{PD}$ denotes Poincar\'{e} duality and $e$ denotes the Euler class.
\end{thm}

\subsubsection{Orientations}
To really make sense of the fundamental class of the moduli space or of the Euler class of the obstruction bundle one needs orientations. We therefore briefly recall how to orient our moduli spaces when they are clean. Recall that $D\dbar|_{W^{1,p}(\Sigma,u^*TX)}$ splits as a sum of its (Fredholm) complex linear and a (compact) complex antilinear parts. We abuse terminology by calling
\[\frac{1}{2}\left(D\dbar(\alpha,\eta,\xi)-\psi D\dbar(\alpha,\eta,\psi\xi)\right)\]
the complex linear part of $D\dbar$ (it is only complex linear in $\xi$).

The linearised Cauchy-Riemann operator is homotopic through Fredholm operators to its complex linear part. There is a {\em determinant bundle} over the space of Fredholm operators whose fibre at $D$ is the determinant line $\Lambda^{\dim\ker(D)}\ker(D)\otimes\Lambda^{\dim\OP{coker}(D)}\OP{coker}(D)$. When the moduli space is regular (so that its tangent space at $u$ is the kernel of $D\dbar$) an orientation of the determinant line is precisely an orientation of the moduli space. Having chosen an orientation on $B$, the determinant line of a complex linear Cauchy-Riemann operator is canonically oriented and one can transport this orientation along a linear homotopy of operators from $D\dbar$ to its complex linear part.

When the moduli space is clean rather than regular an orientation of the determinant line is still all that is needed to define the Euler class of the obstruction bundle.


\section{Two-step nilpotent Lie groups}\label{nilp}


\subsection{Generalities}

A Lie group $N$ is called \emph{$k$-step nilpotent} if its lower central series
\[N\supset [N,N]\supset [N,[N,N]]\supset\cdots\]
reaches the trivial group in $k$ steps. In particular, all iterated Lie brackets of $k+1$ or more elements vanish. We are interested in two-step nilpotent groups. The main advantage of this class is the simplicity of the Baker-Campbell-Hausdorff formula
\[\exp(X)\exp(Y)=\exp\left(X+Y+\frac{1}{2}[X,Y]\right)\]
for the logarithm of a product.

Henceforth, $N$ will denote a connected, simply-connected, two-step nilpotent Lie group of even dimension with Lie algebra $\nn$. For computational convenience we will implicitly embed $N$ into a real linear group $\OP{GL}(V)$ and $\nn$ into $\mathfrak{gl}(V)$ so that we can write $a+b\in\OP{GL}(V)$ for $a,b\in N$ and $XY\in\mathfrak{gl}(V)$ for $X,Y\in\nn$. Note that such an embedding exists by Engel's theorem and that the exponential map $\exp\colon\nn\to N$, which thanks to the embedding in $\OP{GL}(V)$ we can now write
\[\exp(X)=1+X+\frac{1}{2}X^2+\cdots,\]
is a diffeomorphism. We denote its inverse by $\log$.

Since $\nn$ is a linear space there is a canonical isomorphism $T\nn\cong\nn\times\nn$ so we will write $(X,Y)\in\nn\times\nn$ and $Y\in T_{X}\nn$ to mean the same thing. There is a canonical map $\pi_N\colon TN\to\nn$ defined by
\begin{equation}\label{pinbig}
\pi_N(X)=L(s^{-1})_*X\mbox{ for }X\in T_sN
\end{equation}
Here $L(s)\colon N\to N$ is the left-multiplication by $s\in N$. Precomposing with $d\exp\colon T\nn\to TN$ we get a map $\pi_{\nn}\colon\nn\times\nn\to\nn$, explicitly
\begin{equation}\label{pinlittle}\pi_{\nn}(X,Y)=L(\exp(-X))_*(d_{X}\exp)(Y).\end{equation}
\begin{lma}\label{lma:2stepform}
We have
\[\pi_{\nn}(X,Y)=Y-\frac{1}{2}[X,Y].\]
\end{lma}
\begin{proof}
The Baker-Campbell-Hausdorff formula implies that
\[\exp(-X)\exp(X+tY)=\exp\left(t\left(Y-\frac{1}{2}[X,Y]\right)\right)\]
so
\begin{align*}
\lim_{t\to 0}\frac{\exp(X+tY)-\exp(X)}{t}&=\lim_{t\to 0}\frac{1}{t}\exp(X)\left(\exp\left(t\left(Y-\frac{1}{2}[X,Y]\right)\right)-1\right)\\
&=\exp(X)\left(Y-\frac{1}{2}[X,Y]\right).
\end{align*}
In $\OP{GL}(V)$ we know that $L(s)_*v=sv$ so the formula follows.
\end{proof}


\subsection{The twistor family of almost K\"{a}hler structures}

Fix a two-step nilpotent Lie group $N$ as before and endow it with
\begin{itemize}
\item an orientation $\orient$ and
\item a left-invariant metric $g$ (coming from an inner product, also called $g$, on $\nn$).
\end{itemize}
Moreover, let $\Gamma$ be a cocompact lattice in $N$: these always exist if the algebra is defined over $\QQ$.

\begin{dfn}\label{dfn:twist}
The \emph{twistor family}, denoted $\WW$, is the space of pairs $(\omega_{\psi},\psi)$ where
\begin{itemize}
\item $\psi$ is a $g$-orthogonal $\orient$-positive complex structure on $\nn$,
\item $\omega_{\psi}$ is the two-form on $\nn$ associated to $g$ and $\psi$ by $g(v,w)=\omega_{\psi}(v,\psi w)$,
\end{itemize}
and such that
\[\omega_{\psi}([X,Y],Z)+\omega_{\psi}([Y,Z],X)+\omega_{\psi}([Z,X],Y)=0.\]
Any pair $(\omega_{\psi},\psi)\in\WW$ yields a left-invariant almost K\"{a}hler structure $(\Omega_{\psi},J_{\psi})$ on $N$. In particular if $s\in N$, $v\in T_sN$ and $L(s)_*$ denote the differential of left-multiplication by $s$ then
\begin{equation}\label{jdef}J_{\psi}v=L(s)_*\psi L(s^{-1})_*v.\end{equation}
By left-invariance these all descend to give almost K\"{a}hler structures on $\nilman$. We will often abusively write $\psi\in\WW$ or $\omega_{\psi}\in\WW$ or even $J_{\psi}\in\WW$ or $\Omega_{\psi}\in\WW$.
\end{dfn}
This subsumes Definition \ref{torustwist} in the case when $N$ is abelian. Notice that $\WW$ is a subvariety of $SO(2n)/U(n)$, the space of positive orthogonal complex structures on $\nn$, but may not be a smooth subvariety and it may be empty (we will of course restrict attention to examples where it is nonempty!). If it is not smooth we will restrict attention to some auspicious irreducible component of $\WW$ which is smooth.
\begin{lma}\label{twist}
Let $\zz$ denote the centre of $\nn$. For $\psi\in\WW$,
\[\psi[\nn,\nn]\subset\zz^{\perp}\]
If moreover $\zz=[\nn,\nn]\oplus\tt$ for a one-dimensional subalgebra $\tt$ then
\[\psi\zz=\zz^{\perp}.\]
\end{lma}
\begin{proof}
The first assertion follows from the equation
\[g([X,Y],\psi Z)+g([Y,Z],\psi X)+g([Z,X],\psi Y)=0.\]
When we take $\psi[X,Y]\in\psi[\nn,\nn]$ and $Z\in\zz$, the equation reduces to
\[g(\psi[X,Y],Z)=0\]
proving the first claim.

If $\tt$ is one-dimensional then certainly $\psi\tt\subset\tt^{\perp}$. Moreover the first claim implies $\psi\tt\subset[\nn,\nn]^{\perp}$. Hence $\psi\zz=\zz^{\perp}$.
\end{proof}


\subsection{Pseudoholomorphic tori}

We have set up out conventions for coordinates $(p,q)$ and complex structures $j_{\tau}$ on the torus $\torus$ in Section \ref{sct:torus}. We will write $\Delta$ for the Laplacian $\partial_a^2+\partial_b^2$. Notice that if $f\colon\RR^2\to\RR$ is a differentiable function then
\begin{equation}\label{pqabderiv}\partial_af=\partial_pf,\ \partial_bf=\frac{\partial_qf-\tau_1\partial_pf}{\tau_2}\end{equation}


\subsubsection*{The Cauchy-Riemann equations}

Fix a linear complex structure $j_{\tau}$ on $\RR^2$ and let $(a,b)$ be linear conformal coordinates (so $j_{\tau}\partial_a=\partial_b$). Consider $\psi\in\WW$ and the associated left-invariant almost complex structure $J_{\psi}$ on $N$.
\begin{dfn}
A $(j_{\tau},J_{\psi})$-holomorphic torus in a nilmanifold $\nilman$ is a map $u\colon\torus\to\nilman$ such that
\[J_{\psi}\circ du=du\circ j_{\tau}.\]
We will denote by $\pi_1(u)\colon\ZZ^2\to\Gamma$ the induced map on fundamental groups.
\end{dfn}
Note that a $(j_{\tau},J_{\psi})$-holo\-mor\-phic torus in $\nilman$ lifts to a $(j_{\tau},J_{\psi})$-holomorphic map between the universal covers
\[\tilde{u}\colon\RR^2\to N\]
in one of $\Gamma/\pi_1(u)(\ZZ^2)$ possible ways. We will fix one such lift.
\begin{lma}
If $w=\log\circ \tilde{u}\colon\RR^2\to\nn$ then
\begin{equation}\label{pseudohol}\psi\left(\partial_aw-\frac{1}{2}[w,\partial_aw]\right)=\partial_bw-\frac{1}{2}[w,\partial_bw]\end{equation}
which implies
\begin{equation}\label{harmo}\Delta w-\frac{1}{2}[w,\Delta w]=\psi[\partial_a w,\partial_b w]\end{equation}
where $\Delta=\partial_a^2+\partial_b^2$.
\end{lma}
\begin{proof}
The $(j_{\tau},J_{\psi})$-holomorphic map equation
\[J_{\psi}(\tilde{u}(a,b))d\tilde{u}(\partial_a)=d\tilde{u}(j_{\tau}\partial_b)\]
is equivalent to
\[L(\tilde{u}(a,b))_*\psi L(\tilde{u}^{-1}(a,b))_*d\tilde{u}(\partial_a)=d\tilde{u}(\partial_b)\]
because of Equation \eqref{jdef}. This implies that
\[\psi\circ\pi_N\circ d\tilde{u}(\partial_a)=\pi_N\circ d\tilde{u}(\partial_b).\]
(see Equations \eqref{pinbig} and \eqref{pinlittle} for the definition of $\pi_N$ and $\pi_{\nn}$) Splitting
\[\pi_N=\pi_N\circ d\exp\circ d\log\]
and using the fact that $\pi_N\circ d\exp=\pi_{\nn}$ we get a sequence of equations
\begin{align*}
\psi\circ\pi_N\circ d\exp\circ d\log\circ d\tilde{u}(\partial_a)&=\pi_N\circ d\exp\circ d\log\circ d\tilde{u}(\partial_b)\\
\psi\circ\pi_{\nn}\circ dw(\partial_a)&=\pi_{\nn}\circ dw(\partial_b)\\
\psi\circ\pi_{\nn}(w,\partial_aw)&=\pi_{\nn}(w,\partial_bw)
\end{align*}
and this yields \eqref{pseudohol} thanks to Lemma \ref{lma:2stepform}. The second order equation follows by cross-diff\-er\-enti\-at\-ing and manipulating \eqref{pseudohol}.
\end{proof}


\subsection{Homomorphisms}

Let $u\colon\torus\to\nilman$ be a map. The induced map $\pi_1(u)\colon\ZZ^2\to\Gamma$ on fundamental groups extends uniquely to a homomorphism $H\colon\RR^2\to N$. To see this take the images of two generators in $\ZZ^2$: these commute and hence their logarithms commute in the Lie algebra. This means that they span a two-dimensional abelian subalgebra $\RR^2$. The map $H$ is just the exponential map restricted to this subalgebra. Since $H$ sends $\ZZ^2$ into $\Gamma$, it descends to a map $v\colon\torus\to\nilman$.

\begin{lma}
The maps $u$ and $v$ are freely homotopic.
\end{lma}

\begin{proof}
Let $\star$ be a basepoint of $\ZZ^2\backslash\RR^2$. Freely homotoping $u$ using a path $\gamma$ joining $u(\star)$ to $v(\star)$ allows us to assume that $u$ and $v$ are maps based at the same point $u(\star)=v(\star)$. The maps $\pi_1(u)$ and $\pi_1(v)$ on fundamental groups are conjugate by construction and this conjugation can be effected by a further free homotopy of $u$ where the base point traces out a loop based at $u(\star)$. That is, after a free homotopy one can assume $\pi_1(u)=\pi_1(v)$. By a based homotopy one can ensure that the maps $u$ and $v$ agree on the 1-skeleton of $\ZZ^2\backslash\RR^2$ (a wedge of loops). Since $\Gamma\backslash N$ is aspherical (in particular $\pi_2(\Gamma\backslash N)=0$) the homotopy can be extended to the 2-skeleton of $\ZZ^2\backslash\RR^2$.
\end{proof}

\begin{cor}\label{bdd}
Any lift $\tilde{u}\colon\RR^2\to N$ of $u$ differs from $H$ by a bounded amount, i.e. the function $H^{-1}\tilde{u}\colon\RR^2\to N$ given by
\[(p,q)\mapsto H(p,q)^{-1}\tilde{u}(p,q)\]
is bounded.
\end{cor}

\begin{proof}
After perturbation of the projection $u$ to a map $u'$ based at $\star$ and based-homotopic to $v$, the map $H^{-1}\tilde{u}'$ descends to a nullhomotopic map $\torus\to\nilman$ and hence factors as $\RR^2\to\torus\to N$. Boundedness of the map upstairs follows from compactness of $\torus$. The maps $\tilde{u}$ and $\tilde{u}'$ differ by a bounded perturbation (an equivariant lift of a compact perturbation in $\nilman$).
\end{proof}


\section{The Kodaira-Thurston manifold, $K$}\label{KT}


\subsection{Definition}

Consider the two-step nilpotent group
\[N=\left\{\left(\begin{array}{cccc}
1 & x & z & 0 \\
0 & 1 & y & 0 \\
0 & 0 & 1 & 0 \\
0 & 0 & 0 & t
\end{array}\right)\colon x,y,z,t\in\RR,\ t>0\right\}\]
and the lattice $\Gamma$ consisting of matrices with integer entries. The compact quotient $K=\Gamma\backslash N$ is called the Kodaira-Thurston manifold. The Lie algebra $\nn$ consists of matrices
\[\left(\begin{array}{cccc}
0 & x & z & 0 \\
0 & 0 & y & 0 \\
0 & 0 & 0 & 0 \\
0 & 0 & 0 & t
\end{array}\right)\]
and the exponential map is
\[\exp\left(\begin{array}{cccc}
0 & x & z & 0 \\
0 & 0 & y & 0 \\
0 & 0 & 0 & 0 \\
0 & 0 & 0 & t
\end{array}\right)=\left(\begin{array}{cccc}
1 & x & z+\frac{xy}{2} & 0 \\
0 & 1 & y & 0 \\
0 & 0 & 1 & 0 \\
0 & 0 & 0 & e^t
\end{array}\right)\]
The commutator subalgebra $[\nn,\nn]$ consists of matrices with $x=y=t=0$. The centre splits as $\zz=\tt\oplus[\nn,\nn]$ where $\tt=\{x=y=z=0\}$. We pick a basis for $\nn$:
\begin{align*}
\mathbf{n}_1&=\partial_y & \mathbf{n}_2&=\partial_x\\
\mathbf{n}_3&=\partial_t & \mathbf{n}_4&=\partial_z.
\end{align*}

Let us denote by $\WW$ the twistor family. If $\psi\in\WW$ then by Lemma \ref{twist} we know that $\psi(\zz)\subset\zz^{\perp}$. The complex structure $\psi$ is therefore specified by an isometry $\Psi\colon \zz\to \zz^{\perp}$ which we will think of as a two-by-two special orthogonal matrix (written with respect to the bases $\mathbf{n}_3,\mathbf{n}_4$ and $\mathbf{n}_1,\mathbf{n}_2$). It is not hard to check that any matrix $\Psi\in SO(2)$ gives an element $\psi\in\WW$. We will write $\psi_{\theta}$ for the almost complex structure corresponding to the matrix
\[\Psi_{\theta}=\left(\begin{array}{cc}
\cos\theta & -\sin\theta\\
\sin\theta & \cos\theta
\end{array}\right)\]
The tangent space $T_{\psi_{\theta}}\WW$ consists of matrices of the form
\[\alpha=r\left(\begin{array}{cc}
-\sin\theta & -\cos\theta\\
\cos\theta & -\sin\theta
\end{array}\right).\]


\subsection{Pseudoholomorphic tori and homomorphisms}

Given $\psi\in\WW$, suppose that $u\colon\torus\to K$ is a $(j,J_{\psi})$-holomorphic curve for some linear complex structure $j$ on $\RR^2$. Take $(a,b)$ to be linear $j$-complex coordinates on $\RR^2$. Let $\pi_1(u)$ be the induced map on the fundamental group, let $H$ be a homomorphism $\RR^2\to N$ extending $\pi_1(u)$ and let $\tilde{u}\colon\RR^2\to N$ be a lift of $u$. Denote the logarithms by $w=\log\circ\tilde{u}$ and $h=\log\circ H$. We want to compare $\tilde{u}$ and $H$ so consider $C=\log\circ(H^{-1}\tilde{u})$. By the Baker-Campbell-Hausdorff formula
\[C=w-h-\frac{1}{2}[h,w]\]
Moreover since $h$ is a homomorphism its logarithm is linear (a homomorphism of Lie algebras) and hence $\Delta h=0$.

\begin{prp}\label{classiftori}
The logarithm $C$ is constant and hence $\tilde{u}=H\exp(C)$ is a right-translate in $N$ of a homomorphism $\RR^2\to N$. In particular
\[\tilde{u}=\exp\left(h+C+\frac{1}{2}[h,C]\right)\]
\end{prp}

\begin{proof}

We decompose the Lie algebra $\nn$ as $\bb\oplus[\nn,\nn]\oplus\psi[\nn,\nn]$. Note that both $\pp=[\nn,\nn]$ and $\qq=\psi[\nn,\nn]$ are one-dimensional. We denote the corresponding components of $w$ as $w_{\bb},w_{\pp},w_{\qq}$. We have
\begin{align*}
C_{\bb}&=w_{\bb}-h_{\bb}\\
C_{\qq}&=w_{\qq}-h_{\qq}\\
C_{\pp}&=w_{\pp}-h_{\pp}-\frac{1}{2}[h,w]
\end{align*}

\par Equation \eqref{harmo} breaks up into component equations
\begin{align}
\label{CR1}\Delta w_{\bb}&=0\\
\label{CR2}\Delta w_{\qq}&=\psi[\partial_aw,\partial_bw]\\
\label{CR3}\Delta w_{\pp}&=\frac{1}{2}[w,\Delta w]
\end{align}
Equation \eqref{CR1} implies $\Delta C_{\bb}=\Delta w_{\bb}-\Delta h_{\bb}=0$ and because $C_{\bb}$ is bounded the maximum principle tells us that $C_{\bb}$ is constant. Hence $\partial w_{\bb}=\partial h_{\bb}$ is constant (where $\partial$ stands for either $\partial_a$ or $\partial_b$).

Equation \eqref{CR2} implies
\begin{align*}
\Delta C_{\qq}&=\Delta w_{\qq}-\Delta h_{\qq}\\
&=\psi[\partial_aw,\partial_bw]
\end{align*}
We can expand $w=w_{\bb}+w_{\pp}+w_{\qq}$ in the bracket and ignore the $\pp$-components since $\nn$ is two-step nilpotent. Furthermore, the term $[\partial_aw_{\qq},\partial_bw_{\qq}]$ vanishes because $\qq$ is one-dimensional and hence abelian. The term $[\partial_aw_{\bb},\partial_bw_{\bb}]=[\partial_ah_{\bb},\partial_bh_{\bb}]$ is constant. The remaining terms are $[\partial_aw_{\qq},\partial_bh_{\bb}]+[\partial_ah_{\bb},\partial_bw_{\qq}]$ which are linear first order differential operators with constant coefficients acting on the function $w_{\qq}=C_{\qq}+h_{\qq}$. Therefore
\[\Delta C_{\qq}=\psi\left([\partial_a(C_{\qq}+h_{\qq}),\partial_bh_{\bb}]+[\partial_ah_{\bb},\partial_b(C_{\qq}+h_{\qq})]+[\partial_ah_{\bb},\partial_bh_{\bb}]\right)\]
is a linear elliptic equation with constant coefficients for $C_{\qq}$. Boundedness of $C_{\qq}$ and the Hopf maximum principle \cite[Theorem 3.1]{GT} implies that $C_{\qq}$ is constant. The crucial observation is that there are no nonlinearities or couplings in Equation \eqref{CR2} because $\qq$ is one-dimensional.

We now know that $C_{\bb\oplus\qq}=w_{\bb\oplus\qq}-h_{\bb\oplus\qq}$ is constant and, since $h$ is linear $\Delta w_{\bb\oplus\qq}=0$. Equation \eqref{CR3} implies that
\begin{align*}
\Delta w_{\pp}&=\frac{1}{2}[w,\Delta w]\\
&=\frac{1}{2}[w_{\bb\oplus\qq},\Delta w_{\bb\oplus\qq}]=0,
\end{align*}
therefore
\begin{align*}
\Delta C_{\pp}&=\Delta w_{\pp}-\Delta h_{\pp}-\frac{1}{2}\Delta[h,w]\\
&=-\frac{1}{2}\Delta[h_{\bb\oplus\qq},w_{\bb\oplus\qq}]\\
&=-\frac{1}{2}\Delta[h_{\bb\oplus\qq},h_{\bb\oplus\qq}+C_{\bb\oplus\qq}]=0
\end{align*}
because $h$ is linear and $[h_{\bb\oplus\qq},h_{\bb\oplus\qq}]=0$. Again the maximum principle implies that $C_{\pp}$ is constant. We have now seen that all components of $C$ are constant.
\end{proof}


\subsection{Cohomology and its automorphisms}\label{coh}

Consider the left-invariant one-forms
\[dy,\ dx,\ dt,\ \gamma=dz-xdy\]
The first three one-forms are closed and we denote their cohomology classes by $\mathbf{e}_1,\mathbf{e}_2,\mathbf{e}_3$ respectively. They span $H^1(K;\ZZ)\subset H^1(K;\RR)$. The following classes span $H^2(K;\ZZ)$:
\begin{align*}
\mathbf{e}_{13}&=[dy\wedge dt] & \mathbf{e}_{23}&=[dx\wedge dt]\\
\mathbf{e}_{14}&=[dy\wedge\gamma] & \mathbf{e}_{24}&=[dx\wedge\gamma]
\end{align*}
Finally, the following classes span $H^3(K;\ZZ)$:
\[\mathbf{e}_{134}=[dy\wedge dt\wedge \gamma],\ \mathbf{e}_{234}=[dx\wedge dt\wedge\gamma],\ \mathbf{e}_{124}=[dy\wedge dx\wedge\gamma]\]
We define the dual bases $E_i\in H_1(K;\ZZ)$, $E_{ij}\in H_2(K;\ZZ)$ and $E_{ijk}\in H_3(K;\ZZ)$ for homology, so for example
\[\int_{E_{ijk}}\mathbf{e}_{\ell mn}=\delta_{i\ell}\delta_{jm}\delta_{kn}\]
and we write $A=\sum A_{ij}E_{ij}$, or frequently
\[[A_{13},A_{23},A_{14},A_{24}],\]
for the components of a homology class $A$.
\begin{rmk}
The symplectic form $\omega_{\theta}$ corresponding to a rotation matrix $\Psi_{\theta}$ is
\[\omega_{\theta}=dt\wedge(\cos\theta\ dx+\sin\theta\ dy)+\gamma\wedge(-\sin\theta\ dx+\cos\theta\ dy)\]
so the $\omega_{\theta}$-symplectic area of $A$ is
\[-\left(\cos\theta\ (A_{23}+A_{14})+\sin\theta\ (A_{24}-A_{13})\right).\]
\end{rmk}
Let $\phi\colon\Gamma\to\Gamma$ be a lattice automorphism. Then by rigidity for nilpotent Lie groups \cite[Theorem 2.7]{LGLA} we know that $\phi$ extends uniquely to an automorphism of $N$. As we observed in the case of the $2n$-torus, the left-invariant metrics $\phi^*g$ and $g$ are isotopic through left-invariant metrics and hence the corresponding twistor families of symplectic forms are deformation equivalent. Deformation invariance of Gromov-Witten invariants applied to Equation \eqref{gwequi} implies
\begin{equation}\label{gwinvariance}\GW_{1,k}(\WW,\phi_*A)=\phi_*\GW_{1,k}(\WW,A)\end{equation}
for any $\phi\in\Aut(\Gamma)$.
\begin{lma}\label{autos}
There is a homomorphism $\OP{SL}(2,\ZZ)\to\Aut(\Gamma)$ which projects to the standard action of $\OP{SL}(2,\ZZ)$ on $\Gamma/Z(\Gamma)\cong\ZZ^2$.
\end{lma}
\begin{proof}
The homomorphism is defined on generators by
\[\sigma_1\colon =\left(\begin{array}{cc}
0 & -1\\
1 & 0
\end{array}\right)\mapsto\phi_1,\ \sigma_2:=\left(\begin{array}{cc}
1 & 1\\
0 & 1
\end{array}\right)\mapsto\phi_2\]
where
\[\phi_1\left(\begin{array}{c}
x\\
y\\
z\\
t
\end{array}\right)=\left(\begin{array}{c}
-y\\
x\\
z-xy\\
t
\end{array}\right)\]
and
\[\phi_2\left(\begin{array}{c}
x\\
y\\
z\\
t
\end{array}\right)=\left(\begin{array}{c}
x+y\\
y\\
z+\frac{y(y+1)}{2}\\
t
\end{array}\right)\]
Since the projection $\Gamma\to\Gamma/Z(\Gamma)$ is given by $(x,y,z,t)\mapsto(x,y)$ we see that the maps induced on the quotient are precisely $\sigma_1$ and $\sigma_2$.
\end{proof}
The action of $\phi_i$ on second homology is:
\begin{equation}\label{homologyaction}(\phi_i)_*[A_{13},A_{23},A_{14},A_{24}]=[\sigma_i(A_{12},A_{23}),\sigma_i(A_{14},A_{24})].\end{equation}


\subsection{The homology classes of tori}

Let $h\colon\RR^2\to\nn$ be a Lie algebra homomorphism and write
\[h(p,q)=\sum_{i=1}^4h_i(p,q)\ \mathbf{n}_i\]
where $h_i(p,q)$ are its linear coordinate functions. Its exponential is
\begin{equation}\label{expcurv}H=\exp(h)=\left(\begin{array}{ccccccc}
1 & h_2 & h_4+\frac{1}{2}h_1h_2 & 0 \\
0 & 1 & h_1 & 0 \\
0 & 0 & 1 & 0 \\
0 & 0 & 0 & e^{h_3}
\end{array}\right)
\end{equation}
Since $\RR^2$ is an abelian Lie algebra,
\begin{equation}\label{abelian}[\partial_ph,\partial_qh]=\partial_{[p}h_2\partial_{q]}h_1=0.\end{equation}
The map $H\colon\RR^2\to N$ will descend to a closed torus in $K$ if
\[\exp(h(1,0)),\exp(h(0,1))\in\Gamma.\]
Equivalently the derivatives $\partial_ph_i,\partial_qh_i$ for $i=1,2,3$, $\partial_ph_4+\frac{1}{2}\partial_ph_1\partial_ph_2$ and $\partial_qh_4+\frac{1}{2}\partial_qh_1\partial_qh_2$ must be integers.

\begin{dfn}
If $H=\exp(h)$ descends to a closed torus with homology class $A$ then we say $h$ represents the homology class $A$ and we write $[h]=A$. Equivalently, if $h=\log\circ H$ where $H\colon\RR^2\to N$ is the unique homomorphic extension of a lattice homomorphism $\rho\colon\ZZ^2\to\Gamma$ then we can write $[\rho]:=[h]$.
\end{dfn}

\begin{lma}\label{homclass}
Let $h\colon\RR^2\to\nn$ be a Lie algebra homomorphism such that $H=\exp\circ h$ descends to a closed genus one curve in $K$. Then $[h]=A$ where
\[A_{ij}=\partial_{[p}h_i\partial_{q]}h_j.\]
\end{lma}
\begin{proof}
By considering closed invariant two-forms on $N$ we get an isomorphism \cite[Corollary 4.7]{LGLA}
\[H^2(K;\RR)\cong H_{\OP{Lie}}^2(\nn).\]
Similarly we have an isomorphism $H^2(T^2;\RR)\cong H_{\OP{Lie}}^2(\RR^2)\cong\RR$. In terms of these isomorphisms the pullback map $H^2(K;\RR)\to H^2(T^2;\RR)$ is just the pullback in Lie algebra cohomology induced by the homomorphism $h$. This pullback is induced by the map
\[\Lambda^2h^{\vee}\colon\Lambda^2\nn^{\vee}\to\Lambda^2\RR^2\]
which simply takes the two-by-two minors of the matrix representing $h$, whose rows are $(\partial_ph_i,\partial_qh_i)$.

The subspace of $\Lambda^2\nn^{\vee}$ spanned by ${\bf e}_{ij}$ consists of Lie cochains and projects isomorphically to $H_{\OP{Lie}}^2(\nn)$. The coefficients $A_{ij}$ of $A$ are precisely the pullbacks of these forms to $H_{\OP{Lie}}^2(\RR)\cong\RR$ and these are just the minors of the transpose of the matrix whose columns are $\partial_ph$ and $\partial_qh$.
\[A_{ij}=h^*[{\bf e}_{ij}]=\partial_{[p}h_i\partial_{q]}h_j.\]
\end{proof}
Note that it is not immediately obvious why $\partial_{[p}h_1\partial_{q]}h_4$ is an integer (though it follows from the lemma).
\begin{lma}\label{pluck}
If $A=[h]$ for some homomorphism $h\colon\RR^2\to\nn$ then
\[A_{13}A_{24}=A_{14}A_{23}.\]
\end{lma}
\begin{proof}
There is a commutative diagram of Pl\"{u}cker maps
\[\begin{CD}
\RR^2\times\RR^2 @>{\wedge}>> \Lambda^2\RR^2\cong\RR\\
@V{h\times h}VV @VV{\Lambda^2h}V\\
\nn\times\nn @>>{\wedge}> \Lambda^2\nn
\end{CD}\]
The image of $(\partial_p,\partial_q)\in\RR^2\times\RR^2$ in $\Lambda^2\nn$ is the sextuple of two-by-two minors $D_{ij}=\partial_{[p}h_i\partial_{q]}h_j$ of the matrix of $h$. This sits inside the Pl\"{u}cker quadric
\[D_{12}D_{34}-D_{13}D_{24}+D_{14}D_{23}=0\]
However, $D_{12}=0$ because $\RR^2$ is abelian and Lemma \ref{homclass} implies that $D_{ij}=A_{ij}$ for $i=1,2$, $j=3,4$.
\end{proof}

\begin{lma}\label{wlog}
If $h\colon\RR^2\to\nn$ is a homomorphism with $[h]=A$ then there is an automorphism $\phi$ of $\Gamma$ such that $\phi_*A=[m,m,n,n]$ where
\[m=\gcd(A_{13},A_{23}),\ n=\gcd(A_{14},A_{24}).\]
Defining $a=A_{13}/m$, $b=A_{23}/m$, we see from Lemma \ref{pluck} that
\[a=A_{14}/n,\ b=A_{24}/n,\mbox{ and }\gcd(a,b)=1.\]
\end{lma}
\begin{proof}
By the action of $\OP{SL}(2,\ZZ)\subset\Aut(\Gamma)$ on $H_2(K;\ZZ)$ described in Equation \eqref{homologyaction} we can move the pair $(A_{13},A_{23})$ by some $\phi\in\Aut(\Gamma)$ until it coincides with $(\gcd(A_{13},A_{23}),\gcd(A_{13},A_{23}))$. Since $A_{13}A_{24}=A_{14}A_{23}$, the same $\phi$ will take $(A_{14},A_{24})$ to $(\gcd(A_{14},A_{24}),\gcd(A_{14},A_{24}))$.
\end{proof}
We now consider linear reparametrisations of the torus, that is $\OP{SL}(2,\ZZ)$ acting on the $(p,q)$-plane. We have this freedom when counting pseudoholomorphic tori because we specified the complex structure $j_{\tau}$ by giving a point $\tau\in\HH$ in the upper-half plane: we are therefore over-counting each torus infinitely often, once for each point in $\HH$ giving a diffeomorphic complex structure. The $\OP{SL}(2,\ZZ)$-reparametrisation precisely removes this ambiguity. The effect of $\Phi\in\OP{SL}(2,\ZZ)$ on $\sum_{i=1}^2\partial_ph_i\ \mathbf{n}_i$ and $\sum_{i=1}^2\partial_qh_i\ \mathbf{n}_i$ is to act on the right:
\[\left(\begin{array}{cc}
\partial_ph_1 & \partial_qh_1\\
\partial_ph_2 & \partial_qh_2
\end{array}\right)\Phi.\]

\begin{dfn}
We write $[h]_{\OP{SL}}$ for the $\OP{SL}(2,\ZZ)$-equivalence class of Lie algebra homomorphisms containing $h$. Note that the homology class $[h]$ depends only on $h$. The notion of $\OP{SL}(2,\ZZ)$-equivalence also makes sense for lattice homomorphisms $\rho\colon\ZZ^2\to\Gamma$ by extending them to Lie group homomorphisms and taking the logarithm, and we write $[\rho]_{\OP{SL}}$ for the equivalence class.
\end{dfn}

\begin{dfn}\label{reddfn}
We say a homomorphism $h\colon\RR^2\to\nn$ is {\em reduced} if $\partial_ph_1=\partial_ph_2=0$. Equivalently the matrix of derivatives of $h$ is
\[
\left(\begin{array}{cc}
0 & \partial_qh_1\\
0 & \partial_qh_2\\
\partial_ph_3 & \partial_qh_3\\
\partial_ph_4 & \partial_qh_4
\end{array}\right).
\]
We say that $h$ is \emph{fully reduced} if moreover
\[0\leq\partial_qh_3 <\partial_ph_3.\]
The notion of reduced homomorphism also makes sense for lattice homomorphisms $\rho\colon\ZZ^2\to\Gamma$ by extending them to Lie group homomorphisms and taking the logarithm.
\end{dfn}

\begin{lma}\label{reparawlog}
For a homology class $A\neq 0$ with $m=A_{13}=A_{23}$, $n=A_{14}=A_{24}$, any Lie algebra homomorphism $h\colon\RR^2\to\nn$ with $[h]=A$ is $\OP{SL}(2,\ZZ)$-equivalent to a reduced homomorphism with $\partial_qh_1=\partial_qh_2\neq 0$. If moreover $m\neq 0$ then $h$ is $\OP{SL}(2,\ZZ)$-equivalent to a {\em unique} fully reduced homomorphism.
\end{lma}
\begin{proof}
Since $\RR^2$ is an abelian Lie algebra we have
\[0=[\partial_ph,\partial_qh]=\partial_{[p}h_2\partial_{q]}h_1\mathbf{n}_1.\]
If $h_1\not\equiv 0$ this implies that the top two rows
\[
\left(\begin{array}{cc}\partial_ph_1 & \partial_qh_1\\
\partial_ph_2 & \partial_qh_2
\end{array}\right)
\]
of the homomorphism $h$ are linearly dependent. Using the right $\OP{SL}(2,\ZZ)$-action we can ensure that $\partial_ph_1=\partial_ph_2=0$. Now we have
\begin{align*}
m=-\partial_qh_1\partial_ph_3&=-\partial_qh_2\partial_ph_3\\
n=-\partial_qh_1\partial_ph_4&=-\partial_qh_2\partial_ph_4
\end{align*}
and since one of these two quantities is nonzero we know that $\partial_qh_1=\partial_qh_2\neq 0$.

We have a residual right $\OP{SL}(2,\ZZ)$-action by matrices of the form
\[
\left(\begin{array}{cc}
\pm 1 & \star\\
0 & \pm 1
\end{array}\right)
\]
since these preserve the condition $\partial_ph_1=\partial_ph_2=0$. We know that $\partial_ph_3\neq 0$ because $m$ is assumed to be nonzero. Using the action of these matrices and this fact we can attain $\partial_ph_3>0$ and implement the Euclidean algorithm on $\partial_qh_3$ to ensure that $0\leq\partial_qh_3<\partial_ph_3$.
\end{proof}

Note that for a reduced homomorphism, $\partial_qh_1=\partial_qh_2$ divides the greatest common divisor $\gcd(m,n)$. It is easy to check that the most general reduced homomorphism giving the numbers $m$ and $n$ is
\begin{gather*}
\partial_qh_1=\partial_qh_2=-\sgn(m)d\\
\partial_ph_3=\frac{|m|}{d}\\
\partial_ph_4=-\frac{n}{\sgn(m)d}
\end{gather*}
for a positive divisor $d$ of $\gcd(m,n)$, where $\sgn(m)$ denotes the sign of $m$. In matrix form this looks like
\begin{equation}\label{normform}\left(\begin{array}{cc}
0 & -\sgn(m)d\\
0 & -\sgn(m)d\\
\frac{|m|}{d} & \partial_qh_3\\
-\frac{n}{\sgn(m)d} & \partial_qh_4
\end{array}\right).\end{equation}


\section{Pseudoholomorphic tori in $K$}\label{holtor}


We have seen (Proposition \ref{classiftori}) that if $\psi\in\WW$ then all $(j,J_{\psi})$-holomorphic tori in $K$ are quotients of maps $\RR^2\to N$ of the form
\begin{equation}\label{pseudotor}
\exp\left(h+C+\frac{1}{2}[h,C]\right)
\end{equation}
where $h\colon\RR^2\to\nn$ is a Lie algebra homomorphism and $C\in\nn$ is a constant. We know that if $h$ descends to a closed genus one curve then the derivatives
\[\partial_ph_i,\ \partial_qh_i\mbox{ for }i=1,2,3,\ \partial_ph_4+\frac{1}{2}\partial_ph_1\partial_ph_2,\ \partial_qh_4+\frac{1}{2}\partial_qh_1\partial_qh_2\]
are integers. The problem is now to enumerate these tori modulo the repara\-met\-ris\-ation action of $\OP{SL}(2,\ZZ)$ on $\RR^2$.

In light of Lemma \ref{wlog} we will restrict attention to non-zero homology classes with $A_{13}=A_{23}$ and $A_{14}=A_{24}$ without loss of generality and by reparametrising as in Lemma \ref{reparawlog} we can assume $\partial_ph_1=\partial_ph_2=0$ (i.e. $h$ is reduced). Using the usual decomposition $\nn=\aa\oplus\zz$ of the Lie algebra into the centre $\zz$ and its orthogonal complement, we can write this assumption as
\[\partial_ph_{\aa}=0.\]
By Equation \eqref{pqabderiv} this is equivalent to
\[\partial_ah_{\aa}=0.\]

The expected dimension of genus $g$ curves in $K$ which are $J_{\psi}$-holomorphic for some $\psi\in\WW$ is
\[4\cdot(1-g)+6g-6+\dim\WW\]
since $c_1(K)=0$. For tori ($g=1$) the expected dimension is $\dim\WW=1$. When a marked point is added we get
\[\OP{virdim}\mM_{1,1}(\WW,A)=3.\]
Our task in this section is to write down the moduli space and compute its dimension.

\subsection{Defining moduli spaces}

Let $A=[m,m,n,n]\in H_2(K,\ZZ)$ be a non-zero homology class and define the space of maps
\[
\mM(\WW,A)=\left\{(u,\tau,\psi)\ \middle|\ \begin{array}{l}\tau\in\HH,\ \psi\in\WW,\\
u\colon T^2\to K\mbox{ is }(j_{\tau},J_{\psi})\mbox{-holomorphic},\\
u_*[T^2]=A\end{array}\right\}\]
Let $\rho\colon\ZZ^2\to\Gamma$ be a homomorphism with $[\rho]=A$ and let $H\colon\RR^2\to N$ its unique homomorphic extension. Define
\[\mM_{\rho}(\WW)=\{(He^C,\tau,\psi)\in\mM(\WW,A)\mbox{ for some }C\in\nn\}\]
and note that
\[\mM(\WW,A)=\coprod_{[\rho]=A}\mM_{\rho}(\WW)\]
We also define
\begin{gather*}
\mM^{\mathrm{red}}(\WW,A)=\coprod_{[\rho]=A,\ \rho\mbox{ reduced}}\mM_{\rho}(\WW)\\
\mM^{\mathrm{ful}}(\WW,A)=\coprod_{[\rho]=A,\ \rho\mbox{ fully reduced}}\mM_{\rho}(\WW).
\end{gather*} The one-point moduli space is given by
\[\mM_{1,1}(\WW,A)=\mM(\WW,A)\times_{\OP{Aff}(T^2)} T^2\]
where $\OP{Aff}(T^2)=\OP{SL}(2,\ZZ)\ltimes T^2$ is the group of affine reparametrisations of~$T^2$. Here $\phi\in\OP{Aff}(T^2)$ acts by
\[\phi(u,\tau,\psi,z)=(u\circ\phi^{-1},\phi(\tau),\psi,\phi(z))\]
where the action of $\OP{SL}(2,\ZZ)$ on $\HH$ is the standard one. We will also write
\[\mM_{1,1}(\WW,[\rho]_{\OP{SL}}):=\mM_{\rho}(\WW)\times_{T^2}T^2.\]
\begin{lma}\label{noautos}
The action of $\OP{Aff}(T^2)$ on $\mM(\WW,A)\times T^2$ is free.
\end{lma}
\begin{proof}
Suppose $(u,\tau,\psi,z)$ is fixed by $\phi$. Let $v$ denote the simple curve underlying $u$ and $\pi\colon \Sigma_{\tau}\to \Sigma_{\tau'}$ denote the holomorphic covering space such that $u=v\circ\pi$. The curve $v$ has an open set $V\subset\Sigma_{\tau'}$ of injective points. Let $x\in V$. An automorphism $\phi$ of $u$ satisfies $u(\phi x')=u(x')$ for any $x'\in\pi^{-1}(x)$. This implies that $\pi\circ\phi=\pi$ on $\pi^{-1}(V)$ and hence everywhere, so $\phi$ is a deck transformation (a translation). However $\phi(z)=z$, so $\phi=\id$.
\end{proof}
We first divide out by translations. The space $\mM(\WW,A)\times_{T^2}T^2$ has a residual $\OP{SL}(2,\ZZ)$-action. The following lemma is immediate from Lemma \ref{reparawlog}.
\begin{lma}\label{reduction}
Every $\OP{SL}(2,\ZZ)$-orbit of $\mM(\WW,A)\times_{T^2}T^2$ contains a point of
\[\mM^{\mathrm{red}}(\WW,A)\times_{T^2}T^2.\]
If $m\neq 0$ then every orbit contains a unique point of
\[\mM^{\mathrm{ful}}(\WW,A)\times_{T^2}T^2.\]
\qed
\end{lma}

\subsection{Describing $\mM_{\rho}(\WW,A)$}

\begin{lma}\label{redeq}
Suppose that $h$ is a reduced Lie algebra homomorphism and $C\in\nn$ is a constant. The Cauchy-Riemann equations for
\[\exp\left(h+C+\frac{1}{2}[h,C]\right)\]
become
\begin{align}
\label{reducedCR1}\partial_bh_{\aa}&=\psi\partial_ah_{\zz}\\
\label{reducedCR2}\partial_bh_{\zz}&=[C_{\aa},\partial_bh_{\aa}].
\end{align}
\end{lma}
\begin{proof}
By Lemma \ref{twist}, $\psi(\aa)=\zz$. If $w=h+C+\frac{1}{2}[h,c]$ then we have
\begin{align*}
w_{\aa}&=h_{\aa}+C_{\aa}\\
w_{\zz}&=h_{\zz}+C_{\zz}+\frac{1}{2}[h_{\aa},C_{\aa}]
\end{align*}
Taking the $\aa$-part of the Cauchy-Riemann equation \eqref{pseudohol} and using the fact that $\partial_ph_{\aa}=\partial_ah_{\aa}=0$ ($h$ is reduced) gives Equation \eqref{reducedCR1}. Taking the $\zz$-part of \eqref{pseudohol} gives
\begin{align*}0&=\partial_bh_{\zz}+\frac{1}{2}[\partial_bh_{\aa},C_{\aa}]-\frac{1}{2}[h_{\aa}+C_{\aa},\partial_bh_{\aa}]\\
&=\partial_bh_{\zz}+[\partial_bh_{\aa},C_{\aa}].\end{align*}
This last step uses the fact  that $[h_{\aa},\partial_bh_{\aa}]=0$ which follows because $h_{\aa}=a\partial_ah_{\aa}+b\partial_bh_{\aa}=b\partial_bh_{\aa}$.
\end{proof}

\begin{lma}\label{enum}
Fix a Lie algebra homomorphism $h\colon\RR^2\to\nn$ with $h=\sum_{i=1}^4h_i\mathbf{n}_i$ such that $\partial_ph_{\aa}=0$ and $\exp(h(\ZZ^2))\subset\Gamma$. We will list all possibilities for $C$, $\tau$ and $\psi$ such that
\[h+C+\frac{1}{2}[h,C]\]
is the logarithm of a $J_{\psi}$-holomorphic torus:
\begin{itemize}
\item The constant $C_{\zz}$ is arbitrary,
\item The number $\tau_2$ satisfies
\[\tau_2=\frac{||\partial_qh_{\aa}||}{||\partial_ph_{\zz}||}\]
\item The complex structure $\psi$ is specified by the unique matrix $\Psi\in SO(2)$ which rotates
\begin{equation}\label{psiequn}\left(\begin{array}{c}
\partial_ph_3\\
\partial_ph_4
\end{array}\right)\mbox{ to }\frac{1}{\tau_2}\left(\begin{array}{c}
\partial_qh_1\\
\partial_qh_2
\end{array}\right)\end{equation}
\item The components of $C_{\aa}$ satisfy
\begin{equation}\label{consteq}C_2\partial_qh_1-C_1\partial_qh_2=\partial_qh_4-\tau_1\partial_ph_4,\end{equation}
\end{itemize}
Moreover,
\begin{itemize}
\item if $\partial_ph_3\neq 0$ we have
\[\tau_1=\partial_qh_3/\partial_ph_3\]
\item if $\partial_ph_3=0$ then necessarily $\partial_qh_3=0$ and $\tau_1$ is arbitrary.
\end{itemize}
\end{lma}
\begin{proof}
The constant $C_{\zz}$ is arbitrary because it does not enter into Equations \eqref{reducedCR1} and \eqref{reducedCR2}. Taking the norm of Equation \eqref{reducedCR1} gives
\[||\partial_bh_{\aa}||=||\partial_ah_{\zz}||\]
because $\psi$ is orthogonal. Using Equation \eqref{pqabderiv} we have
\[||\partial_bh_{\aa}||=\left|\left|\frac{\partial_qh_{\aa}-\tau_1\partial_ph_{\aa}}{\tau_2}\right|\right|=\frac{||\partial_qh_{\aa}||}{\tau_2}\]
since $\partial_ph_{\aa}=\partial_ah_{\aa}=0$. This gives the formula for $\tau_2$. Having fixed $\tau_2$, Equation \eqref{reducedCR1} becomes precisely the desired condition on $\psi$.

The equation for $C_1$ and $C_2$ is simply the $\mathbf{n}_4$-component of Equation \eqref{reducedCR2}. Finally, the $\mathbf{n}_3$ component is
\[\partial_bh_3=0\]
since $\mathbf{n}_3$ is orthogonal to the commutator subalgebra. Using the fact \eqref{pqabderiv} that
\[\partial_bh_3=\frac{\partial_qh_3-\tau_1\partial_ph_3}{\tau_2}\]
we obtain the required dichotomy for $\tau_1$ when $\partial_ph_3$ is either zero or nonzero.
\end{proof}

Note that, unlike $C_{\aa}$, the quantity $C_2\partial_qh_1-C_1\partial_qh_2$ is invariant under translations of the $(p,q)$-plane.
\begin{cor}\label{generaleqn}
Let $\rho\colon\ZZ^2\to\Gamma$ be a homomorphism, $H$ its homomorphic extension to $\RR^2\to N$ and $h=\log H$. When $\partial_ph_3\neq 0$ the moduli space $\mM_{\rho}(\WW)$ consists of maps $u$ of the form
\[He^{C_0+D}\]
where
\begin{equation}\label{c0def}C_0=\frac{\partial_qh_4-\tau_1\partial_ph_4}{\partial_qh_1}\ {\bf n}_2\end{equation}
and $D$ is any element $\nn$ satisfying $[D,h]=0$.
\end{cor}
\begin{proof}
By Equation \eqref{consteq} we know that if $He^C$ is in the moduli space $\mM_{\rho}(\WW)$ then $C$ must solve
\[C_2\partial_qh_1-C_1\partial_qh_2=\partial_qh_4-\tau_1\partial_ph_4,\]
The vector $C_0\in\nn$ given in Equation \ref{c0def} is a particular solution of this inhomogeneous equation. Therefore $D=C-C_0$ satisfies the corresponding homogeneous equation,
\[D_2\partial_qh_1-D_1\partial_qh_2=0\]
which is equivalent to $[D,h]=0$. Note that $\psi$ and $\tau$ are determined by the homomorphism $\rho$.
\end{proof}

\begin{lma}\label{translwlog}
Suppose $\rho\colon\ZZ^2\to\Gamma$ is reduced and that $(He^{C_0+D},\tau,\psi)\in\mM_{\rho}(\WW)$ as in Corollary \ref{generaleqn}. By reparametrising $(p,q)\mapsto (p+\delta_p,q+\delta_q)$ we can assume that
\[D\cdot\partial_qh_{\aa}=D\cdot\partial_ph_{\zz}=0.\]
In particular we can ensure that $D_{\zz}\perp\partial_ph_{\zz}$ and $D_{\aa}=0$.
\end{lma}
\begin{proof}
The logarithm of $He^{C_0+D}$ is $h+C_0+D+\frac{1}{2}[h,C_0]$ since $[h,D]=0$. We can absorb the reparametrisation of $h$ by $(p,q)\mapsto(p+\delta_p,q+\delta_q)$ into the constant $D$, which becomes
\[D+\delta_p\left(\partial_ph+\frac{1}{2}[\partial_ph,C_0]\right)+\delta_q\left(\partial_qh+\frac{1}{2}[\partial_qh,C_0]\right)\]
First pick $\delta_q$ to solve the equation
\[\delta_q\left(\partial_qh+\frac{1}{2}[\partial_qh,C_0]\right)\cdot\partial_qh_{\aa}+D\cdot\partial_qh_{\aa}=0\]
This is possible since $\left(\partial_qh+\frac{1}{2}[\partial_qh,C_0]\right)\cdot\partial_qh_{\aa}=||\partial_qh_{\aa}||^2\neq 0$ and it ensures that $D'\cdot\partial_qh_{\aa}=0$, where $D'$ is the new constant after the reparametrisation by $(0,\delta_q)$. Next, remember that since $\rho$ is reduced $\partial_ph_{\aa}=0$. Let $\delta_p$ be the solution of
\[\delta_q\left(\partial_ph+\frac{1}{2}[\partial_ph,C_0]\right)\cdot\partial_ph_{\zz}+D'\cdot\partial_ph_{\zz}=0\]
which is possible since $\left(\partial_ph+\frac{1}{2}[\partial_ph,C_0]\right)\cdot\partial_ph_{\zz}=||\partial_{\zz}h_{\zz}||^2\neq 0$. Repara\-met\-ri\-sing $D'$ by $(\delta_p,0)$ gives $D''$ satisfying $D''\cdot\partial_ph_{\zz}=0$. Note that $D''\cdot\partial_qh_{\aa}$ is still zero because $\partial_ph_{\aa}=0$ so this condition is not affected by the second reparametrisation.

Now assume that we have reparametrised and relabelled so that $D$ satisfies the equations
\[D\cdot\partial_qh_{\aa}=D\cdot\partial_ph_{\zz}=0.\]
To see that this gives $D_{\aa}=0$ note that $[D,h]=[D_{\aa},h_{\aa}]=[D_{\aa},q\partial_qh_{\aa}]=0$ (because $h$ is reduced). This is a linear equation for a vector $D_{\aa}\in\aa$ and $\dim(\aa)=2$, so $D_{\aa}$ is a multiple of $\partial_qh_{\aa}$. But we have just seen that $D_{\aa}\perp\partial_qh_{\aa}$. The second equation implies $D_{\zz}\perp\partial_ph_{\zz}$.
\end{proof}

\subsection{Describing $\mM_{1,1}(\WW,A)$: fully reduced case}

We know by definition that $\mM_{1,1}(\WW,A)=\mM(\WW,A)\times_{\OP{Aff}(T^2)}T^2$ and $\mM(\WW,A)=\coprod_{[\rho]=A}\mM_{\rho}(\WW)$. By Lemma \ref{reduction} we know that if $m=A_{13}\neq 0$ then there is a unique fully reduced $\rho$ in the $\OP{SL}(2,\ZZ)$-orbit of homomorphisms representing the class $[A]$ and hence
\begin{align*}
\mM_{1,1}(\WW,A)&=\mM^{\mathrm{ful}}(\WW,A)\times_{T^2}T^2\\
&=\coprod_{\rho\mbox{ reduced}}\mM_{\rho}(\WW)\times_{T^2}T^2
\end{align*}
By Lemma \ref{translwlog} we know that a local slice of the moduli space $\mM_{\rho}(\WW)\times_{T^2}T^2$ (when $\rho$ is fully reduced) is given by
\[(He^{C_0+D(\lambda)},\tau,\psi,z)\]
where $z\in T^2$ is arbitrary and
\begin{equation}\label{dlambda}
D(\lambda)=\lambda\left(\begin{array}{c}
0\\
0\\
\partial_ph_4\\
-\partial_ph_3
\end{array}\right)\end{equation}
for $\lambda\in\RR$. In fact this descends to a global description of the moduli space when we observe that $D(\lambda)$ is central and hence
\[He^{C_0+D(\lambda)}=e^{D(\lambda)}He^{C_0}\]
which gives the same pseudoholomorphic torus if and only if $e^{D(\lambda)}\in\Gamma$. This occurs precisely when $\lambda$ is a multiple of $\frac{1}{\gcd(\partial_ph_3,\partial_ph_4)}$.

\begin{cor}\label{mod-tang}
If $A=\sum A_{ij}E_{ij}\in H_2(K;\ZZ)$ is a homology class with $A_{13}\neq 0$ then the moduli space $\mM_{1,1}(\WW,A)$ is smooth. It is a union of components labelled by fully reduced homomorphisms $\rho\colon\ZZ^2\to\Gamma$, each component consisting of equivalence classes
\[\left[\left(He^{C_0+D(\lambda)},z\right)\right],\ \lambda\in\left[0,\frac{1}{\gcd(\partial_ph_3,\partial_ph_4)}\right],\ z\in T^2\]
where
\begin{itemize}
\item $H\colon\RR^2\to N$ is the unique homomorphic extension of $\rho$ and $h$ is its logarithm,
\item $D(\lambda)$ is defined by Equation \eqref{dlambda},
\item $C_0$ is defined by Equation \eqref{c0def}.
\item the equivalence relation equates $(u,z)$ with $(u\circ\phi^{-1},\phi(z))$ for a translation $\phi\colon T^2\to T^2$ of the domain such that $u\circ\phi^{-1}=u$.
\end{itemize}
In particular, the tangent space at $(u,\tau,\psi,z)$ comprises the vectors
\[(D(\lambda),V)\in\zz\oplus T_z\Sigma_{\tau},\ \lambda\in\RR.\]
The moduli space has dimension three (the expected dimension).
\end{cor}

Arguing as in the proof of Lemma \ref{noautos} we see that if $u$ is a torus and $v$ is the underlying simple torus, so that $u=v\circ\pi$ for some holomorphic covering map $\pi$, then the size of the equivalence class $[(u,z)]$ is the order of the deck transformation group of this cover.

\subsection{Describing $\mM_{1,1}(\WW,A)$: the general case}
Suppose now that $m=A_{13}=0$. Since $A\neq 0$ we know that $n\neq 0$.

By Lemma \ref{reduction} we know that any $\rho$ with $[\rho]=A$ can be conjugated via the action of $\OP{SL}(2,\ZZ)$ to a reduced homomorphism. For a reduced $\rho$, the subgroup $\OP{Stab}(\rho)\subset\OP{Aff}(T^2)$ of affine reparametrisations of $\RR^2$ fixing $\rho$ is generated by the subgroup $T^2$ of translations and the group isomorphic to $\ZZ\times(\ZZ/2\ZZ)\subset\OP{SL}(2,\ZZ)$ consisting of matrices
\[\left(\begin{array}{cc}
\pm 1 & \star\\
0 & \pm 1
\end{array}\right).\]
Lemma \ref{reduction} can be rephrased as
\[\mM_{1,1}(\WW,A)=\mM^{\mathrm{red}}(\WW,A)\times_{\Stab(\rho)}T^2\]
Each component diffeomorphic to
\[\mM_{\rho}(\WW,A)\times_{T^2}T^2\]
for some reduced $\rho$.
\begin{lma}\label{mod-tang2}
In the case $m=0$, $n\neq 0$ the moduli space is a smooth manifold of dimension four and the tangent space at $(u=He^C,\tau,\psi,z)$ comprises triples $(D_3{\bf n}_3,\eta_1,V)\in\zz\oplus\OP{Re}(T_{\tau}\HH)\oplus T_z\Sigma_{\tau}$.
\end{lma}
\begin{proof}
Once again we let $C_0=\frac{\partial_qh_4-\tau_1\partial_ph_4}{\partial_qh_1}$ but remember that in this moduli space $\tau_1$ is allowed to vary so $C_0$ is arbitrary. As in the proof of Lemma \ref{translwlog} we may still reparametrise so that $D_{\aa}=0$ and $D_{\zz}\cdot \partial_ph_{\zz}=0$. Since $m=0$, $\partial_ph_3=0$ and hence $D=D_3{\bf n}_3$. Therefore $\tau_1$, $D_3$, $p$ and $q$ are local coordinates on the moduli space.
\end{proof}

\section{Automorphisms}\label{auto}

We observed in Lemma \ref{noautos} that a holomorphic map $u$ from a torus with one marked point $z$ has no nontrivial holomorphic automorphisms. If we consider only unmarked curves then the automorphism group, $\Aut(u)$, of a multiply-covered curve $u=v\circ\pi$ is precisely the deck transformation group of the holomorphic covering $\pi$. For the one-point moduli space the size of this automorphism group becomes the size of the equivalence classes in Corollary \ref{mod-tang}. Therefore we must now compute $|\Aut(u)|$.

\begin{lma}\label{autosize}
If $A=[m,m,n,n]\in H_2(K;\ZZ)$ is a non-zero homology class and $u=e^he^{C_0+D}$ is a holomorphic torus as in Corollary \ref{generaleqn}, where $h\colon\RR^2\to\nn$ is a Lie algebra homomorphism of the form given in Equation \eqref{normform} then
\[|\Aut(u)|=\gcd(\gcd(m,n),(mk+n\ell)/d)\]
where $k=\partial_qh_4+\frac{\partial_qh_1\partial_qh_2}{2\gcd(\partial_qh_1,\partial_qh_2)}$ and $\ell=\partial_qh_3$.
\end{lma}
\begin{proof}
Suppose that $\pi_1(u)\colon\ZZ^2\to\Gamma$ is the (reduced) homomorphism on fundamental groups. We will write $\pi\colon\Gamma\to\Gamma/Z(\Gamma)\cong\ZZ^2$ for the projection and
\[\pi_1(u)(1,0)=\left(\begin{array}{ccc}
1 & b_1 & d_1\\
0 & 1 & a_1\\
0 & 0 & 1
\end{array}\right)\oplus c_1,\ \pi_1(u)(0,1)=\left(\begin{array}{ccc}
1 & b_2 & d_2\\
0 & 1 & a_2\\
0 & 0 & 1
\end{array}\right)\oplus c_2.\]
Since $\pi_1(u)$ is reduced $a_1=b_1=0$ and $a_2\neq 0$ which implies that $\pi\circ \pi_1(u)$ lands in the cyclic subgroup $\iota\colon\ZZ\hookrightarrow\ZZ^2$ generated by $(\bar{b}_2,\bar{a}_2)$, where
\[\bar{a}_2=\frac{a_2}{\gcd(a_2,b_2)}\quad\mbox{ and }\quad\bar{b}_2=\frac{b_2}{\gcd(a_2,b_2)}.\]
If $v$ is the simple torus underlying $u$ then $\pi\circ\pi_1(v)$ also lands in this subgroup and hence the image of $\pi_1(v)$ is contained in the preimage $\pi^{-1}\iota(\ZZ)$. We have
\[\pi^{-1}(\iota(\ZZ))=\left\{\left.\left(\begin{array}{ccc}1 & q\bar{b}_2 & \ZZ\\
0 & 1 & q\bar{a}_2\\
0 & 0 & 1 \end{array}\right)\oplus \ZZ\right| q\in\ZZ\right\}\cong\ZZ^3\]
The isomorphism with $\ZZ^3$ is
\[(q,r,s)\mapsto\left(\begin{array}{ccc}1 & q\bar{b}_2 & s+\frac{q(q-1)\bar{a}_2\bar{b}_2}{2}\\
0 & 1 & q\bar{a}_2\\
0 & 0 & 1 \end{array}\right)\oplus r\]
Since $\pi_1(u)\colon\ZZ^2\to\pi^{-1}(\iota(\ZZ))\cong\ZZ^3$ is given by the matrix
\[\left(\begin{array}{cc}
0 & \gcd(a_2,b_2)\\
c_1 & c_2\\
d_1 & d_2-\frac{\gcd(a_2,b_2)(\gcd(a_2,b_2)-1)}{2}\bar{a}_2\bar{b}_2
\end{array}\right)\]
the maximal sublattice $\Lambda$ of $\ZZ^3$ (and hence of $\Gamma$) containing $\iota(\ZZ^2)$ as a finite-index sublattice has
\[[\Lambda:\iota(\ZZ^2)]=\gcd(M_1,M_2,M_3)\]
where $M_1,M_2,M_3$ are the two-by-two minors of this matrix. One can see this by putting the matrix into Smith normal form.

In terms of the integer derivatives of the underlying Lie algebra homomorphism,
\begin{align*}
c_1&=\partial_ph_3 & d_1&=\partial_ph_4\\
a_2&=\partial_qh_1 & b_2&=\partial_qh_2\\
c_2&=\partial_qh_3 & d_2&=\partial_qh_4+\frac{1}{2}\partial_qh_1\partial_qh_2
\end{align*}
so $|\Aut(u)|$ is equal to the greatest common divisor of
\[\gcd(\partial_ph_3,\partial_ph_4)\gcd(\partial_qh_1,\partial_qh_2)\]
and
\[\partial_ph_3\left(\partial_qh_4+\frac{\partial_qh_1\partial_qh_2}{2\gcd(\partial_qh_1,\partial_qh_2)}\right)-\partial_ph_4\partial_qh_3.\]
When $h$ has the form given in Equation \eqref{normform}, this expression reduces to
\begin{gather*}\gcd\left(\gcd\left(\frac{m}{d},\frac{m}{d}\right)d,\frac{|m|}{d}k+\frac{n}{\OP{sgn}(m)d}\ell\right)\\
=\gcd(\gcd(m,n),(mk+n\ell)/d)\end{gather*}
where $k=\partial_qh_4+\frac{\partial_qh_1\partial_qh_2}{2\gcd(\partial_qh_1,\partial_qh_2)}$ and $\ell=\partial_qh_3$.
\end{proof}


\section{The linearised Cauchy-Riemann operator}\label{lintheory}
The aim of this section is to prove the following theorem
\begin{thm}\label{cleanthm}
If $A\in H_2(K;\ZZ)$ is a non-zero homology class with $A_{13}=A_{23}$, $A_{14}=A_{24}$ then $\mM_{1,1}(\WW,A)$ is clean (see Definition \ref{clean}). Moreover if $h\colon\RR^2\to\nn$ is a reduced homomorphism (i.e. $\partial_ph_{\aa}=0$) with $[h]=A$ and $\partial_ph_3=0$ (which happens if and only if $A_{13}=0$) then there is a nonvanishing section of the obstruction bundle over the moduli space $\mM_{1,1}(\WW,[h]_{\OP{SL}})$ and hence these holomorphic tori do not contribute to the Gromov-Witten invariant in the class $A$.
\end{thm}
By Lemma \ref{reparawlog} we may always assume that $h$ is reduced.


\subsection{The setup for Fredholm theory}
Fix a homomorphism $\rho\colon\ZZ^2\to\Gamma$. Let $W^{1,\ell}_{\rho}(\RR^2,\nn)$ be the $W^{1,\ell}$-completion of the space of smooth maps $w\colon\RR^2\to\nn$ which are $\rho$-equivariant in the sense that
\[\exp(w(\gamma+z))=\rho(\gamma)\exp(w(z))\ \mbox{for all }\gamma\in\ZZ^2\]
Then $\mB:=\WW\times\HH\times W^{1,\ell}_{\rho}(\RR^2,\nn)$ is a Banach manifold whose tangent space at $(\psi,\tau,w)$ is the vector space
\[T_{\psi}\WW\oplus T_{\tau}\HH\oplus W^{1,\ell}_{\rho}(\RR^2,w^*T\nn)\]
where the subscript $\rho$ denotes equivariant sections. Define the Banach bundle $\mE$ over $\mB$ whose fibre over $(\psi,\tau,w)$ is the $L^\ell$-completion of
\[\Omega^{0,1}_{\psi,\tau,\rho}\left(\RR^2,w^*T\nn\right)\]
where $\Omega^{(0,1)}_{\psi,\tau,\rho}\left(\RR^2,w^*T\nn\right)$ denotes the space of smooth $\rho$-equivariant one-forms on $\RR^2$ with values in $w^*T\nn$ which are anticomplex with respect to the almost complex structures $(j_{\tau},w^*\psi)$. The $\dbar$-operator
\[\dbar(\psi,\tau,w)=\psi\left(\partial_aw-\frac{1}{2}[w,\partial_aw]\right)-\partial_bw+\frac{1}{2}[w,\partial_bw]\]
gives a section of this bundle whose zero-set comprises the logarithms of tori $\torus\to K$ in the given homotopy class which are $(j_{\tau},J_{\psi})$-holo\-mor\-phic for some $\psi\in\WW$ and some $\tau\in\HH$.

The first aim is to understand the kernel of the linearised $\dbar$-operator.
\begin{prp}
The linearised $\dbar$-operator at a pseudoholomorphic torus $(\psi,\tau,w)$ is an operator
\[D_{(\psi,\tau,w)}\dbar\colon T_{\psi}\WW\oplus T_{\tau}\HH\oplus W^{1,\ell}_{\rho}(\RR^2,w^*T\nn)\to L^{\ell}\Omega^{(0,1)}_{\psi,\tau,\rho}\left(\RR^2,w^*T\nn\right)\]
If we define $D(\alpha,\eta,\xi):=D_{(\psi,\tau,w)}\dbar(\alpha,\eta,\xi)(\partial_a)$ (which determines the whole operator $D\dbar$ since this takes values in $(0,1)$-forms) then $D$ is given by the following equation
\begin{align}\label{lineq}
D(\alpha,\eta,\xi)&=\psi\partial_a\xi-\partial_b\xi+\alpha\left(\partial_aw-\frac{1}{2}[w,\partial_aw]\right)+\frac{1}{\tau_2}\left(\eta_1\partial_aw+\eta_2\partial_bw\right)\\
\nonumber&\ \ -\frac{1}{2}\psi\left([\xi,\partial_aw]+[w,\partial_a\xi]\right)+\frac{1}{2}\left([\xi,\partial_bw]+[w,\partial_b\xi]\right).
\end{align}
\end{prp}
\begin{proof}
The equation to-be-linearised is
\[\dbar(\psi,\tau,w)=\psi\left(\partial_aw-\frac{1}{2}[w,\partial_aw]\right)-\partial_bw+\frac{1}{2}[w,\partial_bw]\]
The only nonobvious part of the computation is the effect of an infinitesimal variation $\eta$ of $\tau$
\begin{align*}D(0,\eta,0)&=\frac{1}{\tau_2}\left(\eta_1\partial_aw+\eta_2\partial_bw\right)
\end{align*}
To see this, recall that in the coordinates $(a,b)$ the complex structure is simply $j_i=\Phi_{\tau}j_{\tau}\Phi_{\tau}^{-1}$. If $\eta$ is an infinitesimal variation of $\tau$ then the infinitesimal variation of $j_i$ is computed with respect to the same coordinates $(a,b)$ by
\[\delta_{\eta} j_i=\Phi_{\tau}\delta_{\eta}j_{\tau}\Phi_{\tau}^{-1}\]
since $\Phi_{\tau}$ is just a change of coordinate matrix and hence is not affected by the variation. We compute
\[\delta_{\eta}j_i=\Phi_{\tau}\delta_{\eta}j_{\tau}\Phi_{\tau}^{-1}=\frac{1}{\tau_2}\left(\begin{array}{cc}
-\eta_1 & -\eta_2\\
-\eta_2 & \eta_1
\end{array}\right).\]
\end{proof}


\subsection{Equivariance}\label{equivar}
We now examine more carefully the equivariance condition
\[\exp(w(\gamma+z))=\rho(\gamma)\exp(w(z))\ \mbox{for all }\gamma\in\ZZ^2\]
and its linearisation. By the Baker-Campbell-Hausdorff formula we have
\[w(\gamma+z)=\log\rho(\gamma)+w(z)+\frac{1}{2}[\log\rho(\gamma),w(z)]\]
If $\xi$ is an infinitesimal deformation of $w$ as an equivariant map then
\[\xi(\gamma+z)=\xi(z)+\frac{1}{2}[\log\rho(\gamma),\xi(z)]\]
In particular we see that $\xi_{\aa}$ is $\ZZ^2$-invariant and hence bounded. The combination $\xi_{\zz}-\frac{1}{2}[w,\xi_{\aa}]$ is also $\ZZ^2$-invariant:
\begin{align*}
\xi_{\zz}(\gamma+z)-\frac{1}{2}[w(\gamma+z),\xi_{\aa}(\gamma+z)]&=\xi_{\zz}(z)+\frac{1}{2}[\log\rho(\gamma),\xi_{\aa}(z)]\\
&\ \ \ \ \ \ -\frac{1}{2}[\log\rho(\gamma)+w(z),\xi_{\aa}(z)]\\
&=\xi_{\zz}(z)-\frac{1}{2}[w(z),\xi_{\aa}(z)].
\end{align*}


\subsection{Regularity and obstructions}\label{regobs}

Remember that $w=h+C+\frac{1}{2}[h,C]$ where $h$ is a Lie algebra homomorphism and $C$ is a constant (in particular, second derivatives of $w$ vanish). Cross-differentiating
\[D(\alpha,\eta,\xi)=0\]
using Equation \eqref{lineq} we get
\begin{align*}
0&=\psi\partial_a^2\xi-\partial_a\partial_b\xi-\frac{1}{2}\psi[w_{\aa},\partial_a^2\xi_{\aa}]+\frac{1}{2}([\partial_{[a}\xi_{\aa},\partial_{b]}w_{\aa}]+[w_{\aa},\partial_a\partial_b\xi_{\aa}])\\
0&=\psi\partial_a\partial_b\xi-\partial_b^2\xi+\frac{1}{2}\psi([\partial_{[a}\xi_{\aa},\partial_{b]}w_{\aa}]-[w_{\aa},\partial_a\partial_b\xi_{\aa}])+\frac{1}{2}[w_{\aa},\partial_b^2\xi_{\aa}]\\
&\ \ \ \ \ +\frac{1}{2}\alpha([\partial_aw,\partial_bw])
\end{align*}
Since $h$ is a homomorphism from an abelian Lie algebra
\[[\partial_aw,\partial_bw]=[\partial_ah,\partial_bh]=0.\]
The equations then give
\[\Delta\xi-\frac{1}{2}[w_{\aa},\Delta\xi_{\aa}]-\psi[\partial_{[a}\xi_{\aa},\partial_{b]}w_{\aa}]=0\]
The $\aa$- and $\zz$-parts of this equation are
\begin{align*}
\Delta\xi_{\aa}&=\psi[\partial_{a}\xi_{\aa},\partial_{b}w_{\aa}]\\
\Delta\xi_{\zz}&=\frac{1}{2}[w_{\aa},\Delta\xi_{\aa}]
\end{align*}

\begin{prp}
Suppose that $\psi\in\WW$ and that $w=h+C+\frac{1}{2}[h,C]$ is the logarithm of a $(j_{\tau},J_{\psi})$-holomorphic curve in $K$ with linearised Cauchy-Riemann operator $D$. Then, if $D(\alpha,\eta,\xi)=0$ then $\xi_{\aa}$ and $\xi_{\zz}-\frac{1}{2}[w,\xi_{\aa}]$ are constant.
\end{prp}

\begin{proof}
Split the Lie algebra $\nn$ as in the proof of Proposition \ref{classiftori} into $\bb\oplus[\nn,\nn]\oplus\psi[\nn,\nn]$. We see that
\[\Delta\xi_{\bb}=0\]
and $\xi_{\bb}$, being periodic, is constant. Next
\[\Delta\xi_{\qq}=\psi[\partial_{a}\xi_{\qq},\partial_{b}w_{\bb}]\]
which is a linear elliptic equation with constant coefficients for the single bounded quantity $\xi_{\qq}$, which is therefore constant. Finally
\[\Delta\xi_{\pp}=\frac{1}{2}[w_{\bb},\Delta\xi_{\qq}]=0\]
and so $\xi_{\pp}-\frac{1}{2}[w,\xi_{\aa}]$ is harmonic and bounded.
\end{proof}

To prove Theorem \ref{cleanthm} we compute the kernel of $D$ and then compare with the computation of the tangent spaces of the moduli space in Corollary \ref{mod-tang}. By Lemma \ref{reparawlog}, we assume without loss of generality that $\partial_ph_{\aa}=0$. Taking the $\aa$-part of the linearised equation and using the facts that
\begin{align*}
\alpha(\zz)&\subset\aa&\alpha(\aa)&\subset\zz\\
\partial_a\xi_{\aa}&=0 & \partial_b\xi_{\aa}&=0\\
\partial_a\xi_{\zz}&=\frac{1}{2}\left([\partial_aw_{\aa},\xi_{\aa}]+[w_{\aa},\partial_a\xi_{\aa}]\right)&\partial_b\xi_{\zz}&=\frac{1}{2}\left([\partial_bw_{\aa},\xi_{\aa}]+[w_{\aa},\partial_b\xi_{\aa}]\right)\\
&=0&&=\frac{1}{2}[\partial_bh_{\aa},\xi_{\aa}]
\end{align*}
we get
\[0=\alpha(\partial_ah_{\zz})+\frac{\eta_2}{\tau_2}\partial_bh_{\aa}\]
and for the $\zz$-part we get
\begin{align*}
0&=-\partial_b\xi_{\zz}+\frac{1}{2}[\xi_{\aa},\partial_bh_{\aa}]+\frac{1}{\tau_2}\left(\eta_1\partial_ah_{\zz}+\eta_2(\partial_bw)_{\zz}\right)\\
&=[\xi_{\aa},\partial_bh_{\aa}]+\frac{1}{\tau_2}\left(\eta_1\partial_ah_{\zz}+\eta_2(\partial_bw)_{\zz}\right)
\end{align*}

\begin{lma}\label{kernli}
We describe the kernel of $D$. If $D(\alpha,\eta,\xi)=0$ then the vector $\xi_{\zz}-\frac{1}{2}[w,\xi_{\aa}]$ is arbitrary, $\eta_2=0$ and $\alpha=0$. Moreover:
\begin{itemize}
\item If $\partial_ph_3\neq 0$ then $\eta_1=0$ and $\xi\in\ker[\partial_bh_{\aa},\cdot]$.
\item If $\partial_ph_3=0$ then $\eta_1$ is arbitrary and $\xi_{\aa}$ satisfies
\[[\xi_{\aa},\partial_qh_{\aa}]+\eta_1\partial_ah_{\zz}=0.\]
\end{itemize}
\end{lma}
\begin{proof}
To see $\eta_2=0$, recall from Equation \eqref{reducedCR1} that $\psi\partial_ah_{\zz}=\partial_bh_{\aa}$ so, if
\[\alpha=r\left(\begin{array}{cc}
-\sin\theta & -\cos\theta\\
\cos\theta & -\sin\theta
\end{array}\right)\in T\WW\]
then
\[r\left(\begin{array}{cc}
-\sin\theta & -\cos\theta\\
\cos\theta & -\sin\theta
\end{array}\right)\left(\begin{array}{c}
\partial_ph_3\\
\partial_ph_4
\end{array}\right)=-\frac{\eta_2}{\tau_2}r\left(\begin{array}{cc}
\cos\theta & -\sin\theta\\
\sin\theta & \cos\theta
\end{array}\right)\left(\begin{array}{c}
\partial_ph_3\\
\partial_ph_4
\end{array}\right)\]
Multiplying by $\Psi_{\theta}^{-1}$ on the left tells us that $\left(\begin{array}{c}\partial_ph_3\\ \partial_ph_4\end{array}\right)$ is a real eigenvector of $\left(\begin{array}{cc}0 & -1\\ 1 & 0\end{array}\right)$ unless $r=\eta_2=0$. Since this matrix has only imaginary eigenvalues this is impossible.

The $\zz$-part of the equation now becomes
\[[\xi_{\aa},\partial_bh_{\aa}]+\frac{\eta_1}{\tau_2}\partial_ah_{\zz}=0.\]
Recall from Lemma \ref{redeq} that $\partial_bh_{\zz}=[C_{\aa},\partial_bh_{\aa}]$ so that $\partial_bh_3=0$ (the $t$-direction is orthogonal to the commutator). Since $w_3=h_3+C_3$ and $\mathbf{n}_3\perp[\nn,\nn]$ we see that if $\partial_ph_3\neq 0$ then $\eta_1=0$. The rest is now clear by inspection.
\end{proof}

By comparing with Corollary \ref{mod-tang} and Lemma \ref{mod-tang2} we see that the kernel of $D$ is equal to the tangent space of the moduli space, proving the cleanliness claimed in Theorem \ref{cleanthm}.

From the expected and actual dimension formulae for the moduli spaces we see that the moduli spaces of pseudoholomorphic tori are regular if and only if $\partial_ph_3\neq 0$. We will now write down sections of the obstruction bundles for each moduli space which is not regular. Since a fibre of the obstruction bundle is a space of $(0,1)$-forms, it suffices to specify the value of a section $\sigma$ on the vector $\partial_a$.
\begin{lma}\label{obscalc}
The section $\sigma(\partial_a)=\mathbf{n}_3$ is a nowhere-vanishing section of the obstruction bundle over moduli spaces $h_1\not\equiv 0$, $\partial_ph_3=0$.
\end{lma}
\begin{proof}
Since $\partial_ph_3=0$ and $\partial_qh_3-\tau_1\partial_ph_3=0$ we see that in this case $\partial_qh_3=0$ also. If $(\alpha,\eta,\xi)$ is an infinitesimal variation then we show that
\[\int D_{(\psi,\tau,w)}\dbar(\alpha,\eta,\xi)(\partial_a)\cdot\sigma(\partial_a)\dvol=0\]
by examining the contributions from the three parts separately.

First, since $\alpha(\partial_ah_{\zz})\in\aa$ it is obviously orthogonal to $\mathbf{n}_3$ so this term vanishes. Next, the integrand contribution from $\eta$ is
\[\frac{1}{\tau_2}\left(\eta_1\partial_ah_{\zz}+\eta_2\left(\partial_bh_{\aa}+\partial_bh_{\zz}+\frac{1}{2}[\partial_bh_{\aa},C_{\aa}]\right)\right)\cdot \mathbf{n}_3\]
Since $\partial_ph_3=\partial_qh_3=0$ this vanishes. Finally the contribution from $\xi$ is
\begin{align*}
\int\left(\psi\partial_a\xi-\partial_b\xi-\frac{1}{2}\psi[w_{\aa},\partial_a\xi_{\aa}]+\frac{1}{2}\left([\xi_{\aa},\partial_bh_{\aa}]+[w_{\aa},\partial_b\xi_{\aa}]\right)\right)\cdot \mathbf{n}_3\dvol
\end{align*}
The first two terms vanish by integrating-by-parts (using $\Gamma$-equivariance of $\xi_3$, see Section \ref{equivar}) since $\mathbf{n}_3$ is constant. The other three vanish because $\mathbf{n}_3$ is orthogonal to $\psi[\nn,\nn]$ and to $[\nn,\nn]$.
\end{proof}

This completes the proof of Theorem \ref{cleanthm}.\qed


\subsection{Orientations}\label{orient}

To determine the orientations on our moduli spaces we need to write down a homotopy from the linearised $\dbar$-operator to its com\-plex-linear part $D^{\CC}$. By this we mean the part which is complex-linear in $\xi$. Namely, define
\[\Upsilon(\xi)=\frac{1}{2}\left([\xi,\partial_bw]+[w,\partial_b\xi]-\psi[\xi,\partial_aw]-\psi[w,\partial_a\xi]\right)\]
and set
\begin{align}\label{orient-eq}
D^{\epsilon}(\alpha,\eta,\xi)(\partial_a)&=\psi\partial_a\xi-\partial_b\xi+\alpha\left(\partial_aw-\frac{1}{2}[w,\partial_aw]\right)+\frac{1}{\tau_2}\left(\eta_1\partial_aw+\eta_2\partial_bw\right)\\
\nonumber&\ \ +\Upsilon(\xi)-\frac{\epsilon}{2}(\Upsilon(\xi)-\psi \Upsilon(\psi\xi)).
\end{align}
This is $D$ when $\epsilon=0$ and $D^{\CC}$ when $\epsilon=1$. We obtain (after cross-differen\-tia\-ting)
\begin{align*}
\Delta\xi&=\psi[\partial_{[a}\xi,\partial_{b]}w]+\frac{1}{2}[w,\Delta\xi]-\\
&\ \ \ -\frac{\epsilon}{2}\left(\psi[\partial_{[a}\xi,\partial_{b]}w]+[\psi\partial_{[a}\xi,\partial_{b]}w]\right)+\frac{\epsilon}{4}\left(\psi[w,\psi\Delta\xi]-[w,\Delta\xi]\right)
\end{align*}
This gives $\aa$ and $\zz$ parts
\begin{align}
\label{ORN1}\Delta\xi_{\aa}-\left(1-\frac{\epsilon}{2}\right)\psi[\partial_{[a}\xi_{\aa},\partial_{b]}w_{\aa}]&=\frac{\epsilon}{4}\psi[w_{\aa},\psi\Delta\xi_{\zz}]\\
\label{ORN2}\Delta\xi_{\zz}+\frac{\epsilon}{2}[\psi\partial_{[a}\xi_{\zz},\partial_{b]}w_{\aa}]&=\frac{2-\epsilon}{4}[w,\Delta\xi_{\aa}]
\end{align}
If we define $\orvar=\xi_{\zz}-\frac{1}{2}[w,\xi_{\aa}]$ then
\begin{align*}
\partial \orvar&=\partial\xi_{\zz}-\frac{1}{2}[\partial w,\xi_{\aa}]-\frac{1}{2}[w,\partial\xi_{\aa}]\\
\Delta \orvar&=\Delta\xi_{\zz}-[\partial_aw_{\aa},\partial_a\xi_{\aa}]-[\partial_bw_{\aa},\partial_b\xi_{\aa}]-\frac{1}{2}[w_{\aa},\Delta\xi_{\aa}]
\end{align*}
The left-hand side of Equation \eqref{ORN1} is bounded. The right-hand side can be rewritten as
\[\frac{\epsilon}{4}\psi\left[w_{\aa},\psi\left(\Delta \orvar+[\partial_aw_{\aa},\partial_{a}\xi_{\aa}]+[\partial_bw_{\aa},\partial_{b}\xi_{\aa}]+\frac{1}{2}[w_{\aa},\Delta\xi_{\aa}]\right)\right]\]
which is a sum of terms which are linear or quadratic in $a$ and $b$, in particular it is unbounded unless the coefficients vanish and hence the whole right-hand side is zero. Equation \eqref{ORN1} therefore reduces to a linear elliptic equation which (up to the factor of $1-\frac{\epsilon}{2}$) we have dealt with before. In particular we deduce $\xi_{\aa}$ is constant. Equation \eqref{ORN2} now reduces to a linear elliptic equation we have dealt with before and we deduce that $\orvar$ is constant.


Returning to the original equation \eqref{orient-eq}, and bearing in mind that $\xi_{\aa}$ and $\xi_{\zz}-\frac{1}{2}[w,\xi_{\aa}]$ are constant, we have $\aa$- and $\zz$-components
\begin{align*}
0&=\alpha(\partial_ah_{\zz})+\frac{\eta_2}{\tau_2}\partial_bh_{\aa}+\frac{\epsilon}{4}\psi\left[\psi\left(\xi_{\zz}-\frac{1}{2}[w,\xi_{\aa}]\right),\partial_bw\right]\\
0&=\left(1-\frac{\epsilon}{4}\right)[\xi_{\aa},\partial_bw]+\frac{1}{\tau_2}\left(\eta_1\partial_ah_{\zz}+\eta_2(\partial_bw)_{\zz}\right)
\end{align*}
\begin{lma}
When $\partial_ph_3\neq 0$ the space of solutions to this equation is 3-di\-men\-sional. Explicitly, if we write $\alpha=r\left(\begin{array}{cc}-\sin\theta&-\cos\theta\\ \cos\theta&-\sin\theta\end{array}\right)$, the solutions are:
\begin{align}
\label{Sol1}\tag{S1}\xi_{\aa}&=0,&\orvar&=\partial_ah_{\zz},&r&=0,&\eta&=0\\
\label{Sol2}\tag{S2}\xi_{\aa}&=\partial_bh_{\aa},&\orvar&=0,&r&=0,&\eta&=0
\end{align}
and
\begin{align}
\label{Sol3}\tag{S3}\xi_{\aa}&=-\frac{\partial_ph_4C_{[2}\partial_bh_{1]}}{2\left(1-\frac{\epsilon}{4}\right)|\partial_bh_{\aa}|^2}\left(\begin{array}{c}-\partial_bh_2\\ \partial_bh_1\end{array}\right),&r&=\partial_ph_3,\\
\nonumber \orvar&=\frac{4}{\epsilon}\left(\begin{array}{cc}\cos\theta& -\sin\theta\\ \sin\theta&\cos\theta\end{array}\right)\left(\begin{array}{c}\partial_qh_4\\ -\partial_qh_3\end{array}\right),&\eta&=i\partial_ph_4\tau_2.
\end{align}
\end{lma}
\begin{proof}
The second solution is obvious; the first follows from Equations \eqref{reducedCR1} and \eqref{ORN1}; the third follows from Equations \eqref{reducedCR2} and \eqref{ORN2} and the fact that
\[\left[-\partial_bh_2{\bf n}_1+\partial_bh_1{\bf n}_2,\partial_bh_1{\bf n_1}+\partial_bh_2{\bf n}_2\right]=|\partial_bh_{\aa}|^2{\bf n}_4.\]
\end{proof}

The canonical orientation is now given by picking the oriented basis \eqref{Sol1}, \eqref{Sol2} (which are related by $\psi$, thanks to Equation \eqref{reducedCR1}) and \eqref{Sol3} in that order. When we enumerate the tori we will be able to assume after an $\OP{SL}(2,\ZZ)$-transformation that $\partial_ph_3>0$ so \eqref{Sol3} is positively oriented relative to the base $\WW$. We rescale \eqref{Sol3} by $\epsilon$ and let $\epsilon$ tend to zero. Solution \eqref{Sol3} becomes
\begin{align*}\xi_{\aa}&=0,&\orvar&=4\left(\begin{array}{cc}\cos\theta& -\sin\theta\\ \sin\theta&\cos\theta\end{array}\right)\left(\begin{array}{c}\partial_qh_4\\ -\partial_qh_3\end{array}\right),&r&=0,&\eta&=0.\end{align*}
Note that by \eqref{reducedCR1},
\[\left(\begin{array}{cc}\cos\theta& -\sin\theta\\ \sin\theta&\cos\theta\end{array}\right)\left(\begin{array}{c}\partial_qh_4\\ -\partial_qh_3\end{array}\right)=\left(\begin{array}{c}\partial_ph_4\\ -\partial_ph_3\end{array}\right).\]
By Corollary \ref{mod-tang} and Lemma \ref{translwlog} we know that the unparametrised moduli space $\mM(\WW,A)$, consisting of curves $He^{C_0+D}$, admits a reparametrisation action which one can use to ensure that the \eqref{Sol1} and \eqref{Sol2} components of $D$ vanish so that
\[D=\lambda\left(\begin{array}{c}
0\\
0\\
\partial_ph_4\\
-\partial_ph_3
\end{array}\right).\]
The moduli space $\mM_{1,1}(\WW,A)$ therefore consists of triples $(He^{C_0+D(\lambda)},p,q)$  with $(p,q)\in T^2$. We have shown that the orientation on the moduli space is precisely the one given by the three-form $d\lambda\wedge dp\wedge dq$.


\section{Enumeration of tori in $K$}\label{enumtori}

The aim of this section is to compute the Gromov-Witten invariant in a non-zero homology class $A=\sum A_{ij}E_{ij}$. By Lemma \ref{wlog} we can transform this homology class by an automorphism $\phi$ of $\Gamma$ to a class $\phi_*A=[m,m,n,n]$ where $m=\gcd(A_{13},A_{23})$ and $n=\gcd(A_{14},A_{24})$. Equation \eqref{gwinvariance} implies that
\[\phi_*\GW_{1,1}(\WW,A)=\GW_{1,1}(\WW,\phi_*A)\]
so without loss of generality we can therefore assume that $A=[m,m,n,n]$ for the sake of computing its Gromov-Witten invariants.

For such a class, Theorem \ref{cleanthm} tells us that
\begin{itemize}
\item if $m=0$ then the Gromov-Witten invariant vanishes,
\item if $m\neq 0$ then the moduli space $\mM_{1,1}(\WW,A)$ is regular.
\end{itemize}
We therefore restrict to the case $m\neq 0$. By Lemma \ref{reduction},
\[\mM_{1,1}(\WW,A)=\mM^{\mathrm{ful}}(\WW,A)\times_{T^2}T^2\]
which is a union over all fully reduced homomorphisms $\rho$ with $[\rho]=A$ of
\[\mM_{\rho}(\WW)\times_{T^2}T^2\]
In terms of the Lie algebra homomorphism $h\colon\RR^2\to\nn$ (the logarithm of the unique homomorphic extension $H\colon\RR^2\to N$ of $\rho$), the fully reduced homomorphisms have the matrix form
\[\left(\begin{array}{cc}
\partial_ph_1 & \partial_qh_1\\
\partial_ph_2 & \partial_qh_2\\
\partial_ph_3 & \partial_qh_3\\
\partial_ph_4 & \partial_qh_4
\end{array}\right)=\left(\begin{array}{cc}
0 & -\sgn(m)d\\
0 & -\sgn(m)d\\
\frac{|m|}{d} & \partial_qh_3\\
-\frac{n}{\sgn(m)d} & \partial_qh_4
\end{array}\right).\]
where $d$ is a positive divisor of $\gcd(m,n)$, $\partial_qh_3\in\ZZ$, $0\leq \partial_qh_3<\frac{|m|}{d}$ and $\partial_qh_4+\frac{1}{2}\partial_qh_1\partial_qh_2\in\ZZ$.

We can now use the concrete description of the moduli space given in Corollary \ref{mod-tang} and its orientation as given in Section \ref{orient} to describe the evaluation cycle. The moduli space consists of maps $He^{C_0+D(\lambda)}$ where $H=\exp(h)$ and
\[D(\lambda)=\lambda\left(\begin{array}{c}
0\\
0\\
\partial_ph_4\\
-\partial_ph_3
\end{array}\right),
\ C_0=\left(\begin{array}{c}
0\\
\frac{\partial_qh_4-\tau_1\partial_ph_4}{\partial_qh_1}\\
0\\
0
\end{array}\right),\ \lambda\in\left[0,\frac{1}{\gcd\left(\partial_ph_3,\partial_ph_4\right)}\right].\]
Note that $h$ determines $\tau_1$ and hence also $C_0$. For a Lie algebra homomorphism $h$ and a real number $\lambda$ we denote by $u(\lambda,h)\colon T^2\to K$ the curve represented by
\[He^{C_0+D(\lambda)}\colon \RR^2\to N.\]
\begin{lma}
Let $k\in\ZZ$ and let $\rho\colon\ZZ^2\to\Gamma$ be a reduced homomorphism with underlying Lie algebra homomorphism $h$. Consider $\rho'$, the modified homomorphism whose underlying Lie algebra map $h'$ has the same derivatives as $h$ except that
\[\partial_qh'_4=\partial_qh_4+k.\]
Then the tori $u(\lambda,h)$ and $u(\lambda,h')$ are equal if and only if $k\in(\partial_qh_1)\ZZ=d\ZZ$.
\end{lma}
\begin{proof}
Under this change, $C_0$ changes to $C'_0=C_0+\frac{k}{\partial_qh_1}{\bf n}_2$. We have
\[\exp(h')\exp(C'_0+D(\lambda))=\exp\left(h+qk{\bf n}_4\right)\exp\left(\frac{k}{\partial_qh_1}{\bf n}_2\right)\exp(C_0+D(\lambda))\]
and
\begin{align*}
\exp\left(h+qk{\bf n}_4\right)\exp\left(\frac{k}{\partial_qh_1}{\bf n}_2\right)&=\exp\left(h+qk{\bf n}_4+\frac{k}{\partial_qh_1}{\bf n}_2-\frac{1}{2}qk{\bf n}_4\right)\\
&=\exp\left(h+\frac{k}{\partial_qh_1}{\bf n}_2+\frac{1}{2}qk{\bf n}_4\right)\\
&=\exp\left(\frac{k}{\partial_qh_1}{\bf n}_2\right)\exp(h)
\end{align*}
and this agrees with $\exp(h)$ modulo the right action of $\Gamma$ if and only if $\partial_qh_1$ divides $k$.
\end{proof}
\begin{cor}
Let $A=[m,m,n,n]\in H_2(K;\ZZ)$ be a homology class with $m\neq 0$. For each divisor $d$ of $\gcd(m,n)$ there are $|m|/d$ values of $\partial_qh_3$ and $d$ values of $\partial_qh_4$ giving distinct tori and hence $\mM_{1,1}(\WW,A)$ has $|m|\sigma_0(\gcd(m,n))$ components.\qed
\end{cor}
We need to calculate the homology class of the evaluation cycle for each of these components. For simplicity, we first ignore the equivalence relation $(u,z)\sim (u\circ\phi^{-1},\phi(z))$ for $\phi\in\Aut(u)$ mentioned in Corollary \ref{mod-tang}; this means we are passing to an $|\Aut(u)|$-sheeted cover of the moduli space which we write $\mM'_{1,1}(\WW,A)$. We will later divide out by the size of the automorphism group to compensate for this.

Using the coordinates $(\lambda,p,q)\mapsto\left(He^{C_0+D(\lambda)},p+iq\right)$ on $\mM'_{1,1}(\WW,A)$, the evaluation map sends $(\lambda,p,q)$ to
\[\mbox{\small $
\left(\left(\begin{array}{cccc}
1 & q\partial_qh_2+\frac{\partial_qh_4-\tau_1\partial_ph_4}{\partial_qh_1} & R& 0\\
0 & 1 & q\partial_qh_1 & 0\\
0 & 0 & 1 & 0\\
0 & 0 & 0 & \lambda\partial_ph_4+p\partial_ph_3+q\partial_qh_3
\end{array}\right),\psi\right)\in K\times\WW$}
\]
where $\psi$ is the unique complex structure for which $He^{C_0+D(\lambda)}$ is $(j_{\tau},J_{\psi})$-holomorphic,
\[R=-\lambda\partial_ph_3+p\partial_ph_4+q\partial_qh_4+\frac{1}{2}q^2\partial_qh_1\partial_qh_2\]
and $\lambda\in\left[0,\frac{d}{\gcd(m,n)}\right]$. Since $\psi$ is determined by the derivatives of the Lie algebra homomorphism $h$, see Equation \eqref{psiequn}, it is constant over each component of the moduli space and the evaluation map can be thought of as a 3-cycle in $K$. This cycle represents the three-dimensional homology class
\begin{gather*}-(\partial_qh_1E_{134}+\partial_qh_2E_{234})\frac{|\partial_ph_{\zz}|^2}{\gcd(\partial_ph_3,\partial_ph_4)}\\
=\OP{sgn}(m)\frac{m^2+n^2}{\gcd(m,n)}(E_{134}+E_{234})\end{gather*}
as we can see by integrating the forms $\mathbf{e}_{ijk}$ pulled back along the map $\exp\left(h+C+\frac{1}{2}[h,C]\right)$.

As we remarked above, we are currently overcounting because we have not divided out by the equivalence relation $(u,z)\sim (u\circ\phi^{-1},\phi(z))$ for $\phi\in\Aut(u)$. By Lemma \ref{autosize}, if we write $k=\partial_qh_4+\frac{\partial_qh_1\partial_qh_2}{2\gcd(\partial_qh_1,\partial_qh_2)}$ and $\ell=\partial_qh_3$ then the torus corresponding to the choice of $d$ dividing $\gcd(m,n)$, $0<k\leq d$ and $0<\ell\leq |m|/d$ contributes
\[\frac{1}{\gcd(\gcd(m,n),mk+n\ell)}\]
This gives an factor of
\[(\dagger)=\sum_{d|\gcd(m,n)}\sum_{k=1}^{d}\sum_{\ell=1}^{|m|/d}\frac{1}{\gcd(\gcd(m,n),(mk+n\ell)/d)}.\]
\begin{lma}\label{ickymess}
\[\sum_{d|\gcd(m,n)}\sum_{k=1}^{d}\sum_{\ell=1}^{|m|/d}\frac{1}{\gcd(\gcd(m,n),(mk+n\ell)/d)}=\frac{|m|}{\gcd(m,n)^2}\sum_{d|\gcd(m,n)}d^2.\]
\end{lma}
Before we prove this lemma we give the formula for the 1-point Gromov-Witten invariant $\GW_{1,1}(\WW,A)$ when $A=[m,m,n,n]$:
\begin{gather*}
\frac{m(m^2+n^2)\sigma_2(\gcd(m,n))}{\gcd(m,n)^3}(E_{134}+E_{234})\otimes[\star]\in H_3(K\times\WW;\ZZ)
\end{gather*}
We now ignore the $[\star]$ factor. Pushing this result forward using Equation \eqref{gwinvariance} allows us to compute the Gromov-Witten invariant $\GW_{1,1}(\WW,A)$ for $A=[ma,mb,na,nb]$ where $\gcd(a,b)=1$:
\begin{equation}\label{gwinvt}\boxed{\GW_{1,1}(\WW,A)=\frac{(m^2+n^2)\sigma_2(\gcd(m,n))}{\gcd(m,n)^3}(maE_{134}+mbE_{234}).}\end{equation}
which proves Theorem \ref{ourmainthm}.\qed
\begin{proof}[Proof of Lemma \ref{ickymess}]
For convenience, define $\mu=\gcd(m,n)$, $\bar{m}=|m|/\mu$, $\bar{n}=n/\mu$ and $\lambda=\bar{m}k+\bar{n}\ell$. We have:
\[\frac{1}{\gcd\left(\gcd(m,n),(mk+n\ell)/d\right)}=\frac{d/\mu}{\gcd(d,\OP{sgn}(m)\bar{m}k+\bar{n}\ell)}\]
We convert the sum over $k$ into a sum over $\lambda$:
\begin{align*}
\sum_{k=1}^d\frac{1}{\gcd\left(d,\OP{sgn}(m)\bar{m}k+\bar{n}\ell\right)}&=\sum_{\lambda=1}^d\frac{\#\{k:\OP{sgn}(m)\bar{m}k+\bar{n}\ell\equiv\lambda\mod d\}}{\gcd(d,\lambda)}\\
&=\sum_{\lambda=1}^d\frac{\gcd(\bar{m},d)}{\gcd(d,\lambda)}\OP{if}\left(\gcd(\bar{m},d)|\lambda-\bar{n}\ell\right)
\end{align*}
where $\OP{if}(X)$ is the Boolean function taking the value 1 if $X$ is true and 0 otherwise. To get this line we use the fact that a linear congruence $ax=y\mod d$ has $\gcd(a,d)$ solutions modulo $d$ if $\gcd(a,d)|y$ and none otherwise. Now perform the sum over $\ell$:
\begin{align*}
\sum_{\ell=1}^{|m|/d}\OP{if}\left(\gcd(\bar{m},d)|\lambda-\bar{n}\ell\right)&=\frac{|m|}{d\gcd(\bar{m},d)}\\
&=\frac{\bar{m}}{\gcd(\bar{m},d)}\frac{\mu}{d}
\end{align*}
since $\lambda-\bar{n}\ell\equiv 0\mod \gcd(\bar{m},d)$ has a unique solution modulo $\gcd(\bar{m},d)$, since $\gcd(\bar{m},\bar{n})=1$, and hence $|m|/d\gcd(\bar{m},d)$ solutions in $\{1,\ldots,|m|/d\}$. Substituting this back into the full formula gives
\begin{align*}(\dagger)&=\sum_{d|\mu}\sum_{\lambda=1}^d\frac{\gcd(\bar{m},d)}{\gcd(d,\lambda)}\frac{d}{\mu}\frac{\bar{m}}{\gcd(\bar{m},d)}\frac{\mu}{d}\\
&=\sum_{d|\mu}\sum_{\lambda=1}^d\frac{\bar{m}}{\gcd(d,\lambda)}\\
&=\frac{|m|}{\gcd(m,n)^2}\sum_{d|\gcd(m,n)}d^2\end{align*}
where in the last line we have used Ces\`{a}ro's formula
\[\sum_{d|n}\sum_{i=1}^df(\gcd(i,d))=\sum_{d|n}f\!\left(\frac{n}{d}\right)d,\]
valid for any arithmetic function $f$: this follows from {\cite[page 129]{Dick}} and the elementary properties of Dirichlet convolutions.
\end{proof}


\section{Acknowledgements}

The authors would like to acknowledge helpful discussions with Paul Biran, Rahul Pandharipande and Dietmar Salamon. An anonymous referee was also extremely helpful. J.E. was supported by an ETH Postdoctoral Fellowship. Our collaboration is supported by EPSRC grant EP/I036044/1.

\end{document}